\documentclass[sn-mathphys-num]{sn-jnl}

\usepackage{graphicx}%
\usepackage{multirow}%
\usepackage{amsmath,amssymb,amsfonts}%
\usepackage{amsthm}%
\usepackage{mathrsfs}%
\usepackage[title]{appendix}%
\usepackage{xcolor}%
\usepackage{textcomp}%
\usepackage{manyfoot}%
\usepackage{booktabs}%
\usepackage{algorithm}%
\usepackage{algorithmicx}%
\usepackage{algpseudocode}%
\usepackage{listings}%
\usepackage{caption}    
\usepackage{subcaption} 
\expandafter\def\csname ver@subfig.sty\endcsname{}

\usepackage{float}      
\usepackage{booktabs,makecell}
\usepackage{adjustbox}
\usepackage{diagbox}

\usepackage{xcolor}

\usepackage{graphicx,epstopdf} 

\usepackage{array}

\usepackage{geometry}
\theoremstyle{thmstyleone}%
\newtheorem{theorem}{Theorem}
\newtheorem{Proposition}[theorem]{Proposition}%

\theoremstyle{thmstyletwo}%
\newtheorem{example}{Example}%
\newtheorem{remark}{Remark}%

\theoremstyle{thmstylethree}%
\newtheorem{definition}{Definition}%

\newtheorem{lem}{Lemma}[section]

\begin{document}

\title[Article Title]{Enhancing multigrid solvers for isogeometric analysis of nonlinear problems using polynomial extrapolation techniques}


\author*[1,2]{\fnm{Abdellatif} \sur{Mouhssine}}\email{abdellatif.mouhssine@um6p.ma}

\author[1]{\fnm{Ahmed} \sur{Ratnani}}\email{ahmed.ratnani@um6p.ma}
\equalcont{These authors contributed equally to this work.}

\author[2]{\fnm{Hassane} \sur{Sadok}}\email{hassane.sadok@univ-littoral.fr}
\equalcont{These authors contributed equally to this work.} 
\affil*[1]{\orgdiv{The UM6P Vanguard Center}, \orgname{Mohammed VI Polytechnic University}, \orgaddress{\street{Lot 660 Hay Moulay Rachid}, \city{Benguerir}, \postcode{43150}, 
\country{Morocco}}}

\affil[2]{\orgdiv{Laboratoire de Math\' ematiques Pures et Appliqu\' ees}, \orgname{Universit\'e du Littoral C\^ote d'Opale}, \orgaddress{\street{50 Rue F. Buisson}, \city{Calais Cedex}, \postcode{B.P. 699 - 62228}, 
\country{France}}}


\abstract{When used to accelerate the convergence of fixed-point iterative methods, such as the Picard method, which is a kind of nonlinear fixed-point iteration, polynomial extrapolation techniques can be very effective. The numerical solution of nonlinear problems is further investigated in this study. Particularly, using multigrid with isogeometric analysis as a linear solver of the Picard iterative method, which is accelerated by applying vector extrapolation techniques, is how we address the nonlinear eigenvalue Bratu problem and the Monge–Ampère equation. This paper provides quadratic convergence results for polynomial extrapolation methods. Specifically, a new theoretical result on the correlation between the residual norm and the error norm, as well as a new estimation for the generalized residual norm of some extrapolation methods, are given. We perform an investigation between the Picard method, the Picard method accelerated by polynomial extrapolation techniques, and the Anderson accelerated Picard method. Several numerical experiments show that the Picard method accelerated by polynomial extrapolation techniques can solve these nonlinear problems efficiently.

}



\maketitle
\section{Introduction}\label{sec1}
Since there has been a recent surge of interest in the numerical solution of the nonlinear elliptic partial differential equation (PDE), this paper concerns an efficient method for using advanced computational approaches to solve challenging nonlinear problems such as the Bratu problem \cite{R1}, which is a nonlinear eigenvalue problem, which plays a significant role as a benchmarking tool for evaluating numerical algorithms \cite{R2,R3,R4}. Originating from the solid fuel ignition model in thermal combustion theory, the Bratu equation exemplifies a challenging nonlinear elliptic partial differential equation \cite{R5,R6}.
And the  Monge-Ampère equation \cite{R7, R8, R9, R10} which is a fully nonlinear elliptic PDE.

For the efficient solution of both linear and nonlinear systems of algebraic equations, multigrid methods are a useful tool. Several multigrid techniques are known; refer to \cite{R11,R12,R13}. We investigate the geometric multigrid approaches (GMG) \cite{R14} with isogeometric discretization (IGA) (refer to \cite{R15,R16,R17}) in this study. Convergence acceleration techniques are frequently employed because the convergence of these sequences produced by linear or nonlinear fixed-point iterative methods is too slow. Techniques for vector extrapolation methods were particularly developed to handle these vector sequences. Vector polynomial extrapolation methods \cite{R18,R19}, vector and topological epsilon algorithms \cite{R20,R21}, and Anderson acceleration methods \cite{R22,R23,R24,R25,R26} are some of these techniques. We will restrict our focus in this study to polynomial extrapolation techniques \cite{R14} such as the reduced rank extrapolation method (RRE), which was attributed to Eddy \cite{R27} and Mesina \cite{R28}  and the minimal polynomial extrapolation (MPE) method \cite{R29}.

Regarding the efficiency of the multi-iterative approach \cite{R14} to linear problem solving that combines geometric multigrid methods with IGA and polynomial extrapolation techniques, we investigate in this paper how to solve nonlinear elliptic problems by integrating polynomial extrapolation techniques with geometric multigrid approaches. Specifically focusing on the Bratu problem and the Monge-Ampère equation. Our contribution is a numerical demonstration of the efficiency of vector extrapolation methods. It involves using the Picard iterative method instead of the commonly used Newton variants. We then accelerate the Picard iterations by applying the restarted polynomial extrapolation methods \cite{R14,R18,R30,R31}, either minimal polynomial extrapolation and reduced rank extrapolation.
During each Picard iteration, a multigrid scheme (such as the V-cycle) is employed to solve the linear system using IGA. This combination results in the MPE-Picard-MG and RRE-Picard-MG methods. We perform extensive numerical experiments to validate the efficiency of our proposed techniques, showcasing their potential for enhancing the performance of multigrid solvers for nonlinear problems in the context of IGA. The reason for that is widely known: linking polynomial extrapolation techniques with Picard enhances Picard's convergence, since the convergence of Picard stagnates and, in most cases, diverges. It is evident that the Picard method diverges when the Bratu problem's parameter $\lambda$ is high. In contrast, our approaches MPE-Picard-MG and RRE-Picard-MG converge and are robust with respect to the parameters  $\lambda$ in Bratu's problem, the spline degree $p$ and the  mesh size $h$. Therefore, these combinations can lead to improved, robust convergence results and a large reduction in CPU time when compared to Picard's method without extrapolation (Picard-MG) or even Picard accelerated using Anderson acceleration (AA-Picard-MG) for the resolution of nonlinear problems such as the Bratu problem and the Monge-Ampère equation.

The reminder of this paper is as follows: In Section \ref{sec2}, a summary of isogeometric analysis (IGA) is given, along with a discussion of multigrid schemes. In Section \ref{sec3}, polynomial extrapolation techniques such as reduced rank extrapolation (RRE), minimal polynomial extrapolation (MPE), and other acceleration techniques such as the Anderson acceleration method (AA) are introduced. Their implementations are described in depth. A new estimation for the norm of the generalized residual of some vector extrapolation methods is provided in Theorem~{\upshape\ref{theom1}} and a new theoretical result about the connection between error norm and residual norm for some polynomial extrapolation methods is presented in Theorem~{\upshape\ref{theom3}}. In order to demonstrate the effectiveness of the polynomial extrapolation techniques in improving the convergence rate of Picard's iterations for the resolution of the  Bratu problem and the Monge–Ampère equation, Section \ref{sec4} provides a variety of numerical experiments. The outcomes of these studies involving the RRE-Picard-MG and MPE-Picard-MG approaches are also included in this section, along with a comparison with the Anderson-accelerated Picard method (AA-Picard-MG).

\section{Multigrid methods}\label{sec2}
Iterative numerical techniques known as multigrid methods were first created to solve huge algebraic systems that resulted from the discretization of both linear and nonlinear elliptic partial differential equations \cite{R11,R12,R13}.
Let us consider a linear PDE with Dirichlet boundary conditions. 
\begin{equation}
\begin{array}{rcl}
Lu & = & f, \label{equat1}
\end{array}
\end{equation} 
in a bounded domain $\Omega$, which we will call the unit square for simplicity. The mesh of stepsize $h$ covers the domain $\Omega$. Additionally, assume that the isogeometric analysis approach is used to discretize the equation (\ref{equat1}). This results in a set of linear equations: 
\begin{equation}
\begin{array}{rcl}
A^h u^h & = & F^h, \label{equat2}
\end{array}
\end{equation} 
where $A^h$ is a square nonsingular SPD matrix; it is called the stiffness matrix, $F^h$ is a specified vector that originates from the right-hand side of (\ref{equat1}) and the boundary conditions, and $u^h$ is the vector of unknowns. Let $e^h = u^h- v^h$ be the associated algebraic error, and let $v^h$ be the current approximation of $u^h$ provided by the multigrid iteration. The error is known to satisfy the following residual equation (\ref{equat3}): 
\begin{equation}
\begin{array}{rcl}
A^h e^h & = & r^h, \label{equat3}
\end{array}
\end{equation} 
where $r^h= F^h-A^h v^ h$ is the residual vector. Knowing that the grid $\Omega^h$ of mesh size $h$ is called the ﬁne grid, we deﬁne a coarse grid $\Omega^H$ with mesh size
$2h$. Using multigrid techniques, one may approximate the error on the coarse grid. Therefore, we must solve the coarse grid problem $A^{2h} e^{2h} = r^{2h}$, project the solution onto the fine grid, and update the previous approximation $v^h$. Then $v^h + e^h$
could be a better approximation for the exact solution $u^h$.\\
The mechanism that transforms the data from the coarse grid $\Omega^{2h}$ to the fine grid $\Omega^{h}$ is called the interpolation or prolongation operator $P_{2h} ^{h}$, and the one who transforms the data from the fine grid to the coarse grid is called the restriction operator $R_h ^{2h}$ (which is often chosen as $(P_{2h}^h )^{T}$). These two operators are defined as follows:

$$ P_{2h}^h: \Omega^{2h} \longrightarrow \Omega^{h},$$
and
$$ R_h ^{2h}: \Omega^{h} \longrightarrow \Omega^{2h}.$$
In the multigrid methods, we need to define the coarse grid matrix $A^{2h}$, we can take $A^{2h}$ as the result of the isogeometric discretization to the differential operator on $\Omega^{2h}$, i.e., the $\Omega^{2h}$ version of the initial matrix system's $A^h$. We can also define $A^{2h}$ by the Galerkin projection (it is also called the Galerkin condition): Since the residual equation on the fine grid is given by (\ref{equat3}), we suppose that the error $e^h$ is in the range of the interpolation operator; therefore, there is a vector $e^{2h}$ on $\Omega^{2h}$ such that $e^h= P_{2h}^h e^{2h}$. Hence,
$$A^ h P_{2h}^h e^{2h}= r^h, $$
when we multiply the two sides of this equation by the restriction operator $R_{h}^{2h}$, we obtain
$$  R_{h}^{2h} A^ h P_{2h}^h e^{2h}= R_{h}^ {2h} r^h. $$
If we compare this last equation with $A^{2h} e^{2h} = r^{2h},$ we conclude that the coarse grid matrix is given by (\ref{equat4}):
\begin{equation}
\begin{array}{rcl}
A^{2h} & = & R_{h}^{2h} A^ h P_{2h}^h. \label{equat4}
\end{array}
\end{equation}

\subsection{Description of multigrid algorithms}\label{subsubsec2}
We explain the Two-Grid (TG) version of the multigrid algorithm. Here is a description of the TG algorithm \cite{R11,R12,R13}:

\begin{itemize}
\item Pre-Smooth on $A^hu^h=F^h$.
\item Compute the residual: $r^h= F^h-A^hu^h.$\;\;(Note: $A^he^h=r^h$).
    \item Solve the coarse residual equation: $A^{2h}e^{2h}= R_h^{2h}r^h$.
    \item Correct the fine grid approximation: $u^{h} \longleftarrow u^{h} + P_{2h}^{h}e^{2h} $.
\item Post-Smooth on $A^hu^h=F^h$.
\end{itemize}

The Two-Grid approach is generalized by the multigrid method. Using nested coarser grids that define various levels, the TG's coarse grid problem is solved recursively. As a result, we get the V-Cycle scheme \cite{R11,R12,R13} (Algorithm~\ref{algo1}), one of the families of multigrid schemes. 

  \begin{algorithm}
\caption{V-Cycle algorithm}\label{algo1}
\begin{algorithmic}[1]
\Require $A^{h}, F^{h}, v^{h}, \nu_1, \nu_2, \text{maxiter}$ 
\Ensure $u^{h}$
\State $k \Leftarrow 0$
\While{$k \leq \text{maxiter and not converged}$}
 \If {$\Omega^h$ is the coarsest grid}
            \State $u^h \Leftarrow (A^h)^{-1} F^h$     \Comment{Direct solve on the coarsest grid}
           
        \Else

            \State $u^h \Leftarrow  \text{smoother}(A^h,F^h,v^h,\nu_1)$
            \Comment{Pre-smooth using relaxation}
            \State $r^h \Leftarrow  F^h- A^h u^h$  \Comment{Compute the residual}
            \State $r^{2h}\Leftarrow  R_h^{2h} r^h$  \Comment{Restrict the residual}
            \State $e^{2h} \Leftarrow 0$ \Comment{Initialize coarse grid error}
            \State $e^{2h} \Leftarrow  \text{MGV}(A^{2h},r^{2h}, e^{2h},\nu_1,\nu_2)$  \Comment{V-cycle recursion on coarse grid}
            \State $u^h \Leftarrow  u^h + P_{2h}^{h}e^{2h}$  \Comment{Correct the approximation}
            \State $ u^h \Leftarrow  \text{smoother}(A^h,F^h,u^h,\nu_2)$   \Comment{Post-smooth using relaxation}
\EndIf
\State $k \Leftarrow  k + 1$
\EndWhile
\end{algorithmic}
\end{algorithm}

\subsection{B-splines $\&$ IGA}
In one-dimensional space, a vector of nodes or Knots vector $\xi_1,\xi_2,...,\xi_{N+p+1}$ is non-decreasing
set of coordinates in the parameter space, where $N$ is the number of control points and $p$ is
the degree of the spline. If the nodes $\xi_{i}, \; i=1,...,N+p+1$ are equidistant, we say that this
vector of nodes is uniform. And if the nodes in the first and last position are identical, $p+1$
times, we say that this vector of nodes is open (open Knots vector).

Given $m$ and $p$ as natural numbers, let us consider a sequence of non-decreasing real numbers: $T= \{t_{i}\}_{0\leq i \leq m}$. T is called the knots sequence. From a knots sequence, we can generate a B-Spline family using the recurrence formula (~\ref{moneq}). For more details, see \cite{R10,R15,R32,R33}.
 \begin{definition}[B-Splines using Cox-DeBoor Formula]
 The jth B-spline of degree $p$ is defined by the recurrence relation:
 \begin{equation}
    N^p_j = \frac{t-t_j}{t_{j+p}-t_j}  N^{p-1}_j + \frac{t_{j+p+1}-t}{t_{j+p+1}-t_{j+1}}  N^{p-1}_{j+1},
 \label{moneq}
  \end{equation}   
       where 
       $$ N^0_j(t)= \chi_{[t_j,t_{j+1}]}(t),$$
for $0<j<m-p-1.$           
 \end{definition}
\begin{Proposition}[B-Splines properties]\label{prop3}
\begin{itemize}
\item Compact support. $N^p_j(t)=0 \, \forall t \notin[t_j,t_{j+p+1}).$ 
 \item If $t \in[t_j,t_{j+1}),$ then only the B-splines $\{ N^p_{j-p},..., N^p_j\} $  are non vanishing at t.  
\item Non-negativity. $ N^p_j(t)\geq 0 \,\,\forall t \in [t_j,t_{j+p+1}).$ 
 \item Partition of the unity. $\sum{N^p_i(t)}=1,\, \forall t \in \mathbb{R}.$ 
\end{itemize}      
\end{Proposition}

\begin{proof}
    For the proof of these properties, see \cite{R15}.
\end{proof}

\begin{remark}
\begin{itemize}
\item In both methods (finite element and B-spline), the same variational formulation is used.
\item In the classical finite element method, the Lagrange interpolation functions are chosen,
and in the B-spline method, one chooses the B-splines functions. For degree one, the
basis of the two methods are coincident, but they are different in degree greater than
1.
\item Lagrange interpolation functions of degree $p$ are only of class $C^0$ but the B-spline functions of degree $p$ are of class $C^ {p-1}$.
\item The Lagrange interpolation functions can be given negative values, but the B-splines
functions are always positive, so all the components of the stiffness matrix in the B-spline method are always positive.
\item With the same $N$ elements in the classical elements for degree $p$, one must compute
$N × p+1$ basis functions, but in the B-spline method, one only has to compute $N + p$
basis functions. 
\end{itemize}
\end{remark}

\begin{example}
  Some examples of B-spline basis functions with different degrees $p$ are given in Figure~\ref{bsplines examples}.
  \begin{figure}[ht!]
      \centering
      \includegraphics[width=5cm,height=5cm]{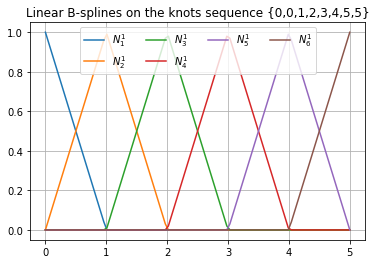}
    \includegraphics[width=5cm,height=5cm]{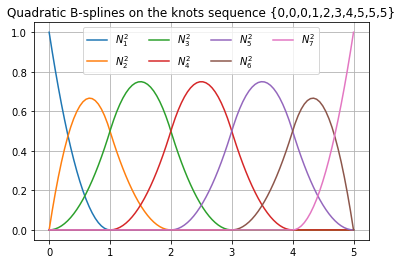}
     \includegraphics[width=5cm,height=5cm]{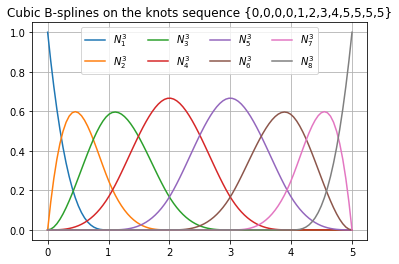}
      \caption{B-spline basis functions of order $p=1,2,3$}
      \label{bsplines examples}
  \end{figure}   
\end{example}

\subsection{Picard iterative method for nonlinear problems}
\subsubsection{The Bratu problem}
Consider the $2D$ Bratu problem (\ref{eq2dbratu}) on a bounded domain $\Omega$ where $f\in L^{2}(\Omega)$:
\begin{equation}
 \begin{cases}
-\text{div}\;(\mathbf{\nabla}u) +\lambda e^{u}& = f, \;\;\;\;\;\;\;  \text{in} \;\; \Omega, \\
\;\;\;\;\;  u &= 0, \;\;\;\;\;\;\; \text{on} \;\;\; \Gamma=\partial \Omega,
\end{cases}
\label{eq2dbratu}
\end{equation}
We multiply both sides of the equation by a test function $v\in V= H_{0}^{1}(\Omega)$ and by using the Green formula \cite{R17} we obtain:
    $$\int_{\Omega }-\Delta u^{n}.v = \int_{\Omega }f v - \lambda \int_{\Omega }e^{u^{n-1}}v,$$
    which gives
    $$ \int_{\Omega }\nabla u^{n}.\nabla v= \int_{\Omega }(f- \lambda e^{u^{n-1} })v.$$
We consider the bilinear form: 
$$a(u^n,v)= \int_{\Omega }\nabla u^{n}.\nabla v,$$
and the linear form:
$$L(v)= \int_{\Omega }(f- \lambda e^{u_{n-1} })v.$$
Then, the associated variational formulation is:
 $$ \begin{cases}
\text{Find} \; u^n \in V= H_{0}^{1}(\Omega) \;\text{such that }   \\
 a(u^n,v)= L(v), \;\text{for all} \; v\in V. \\
  \end{cases}$$
Let $V_{h}=\textit{span}\{B_{i}^{p}/i=2,...N-1\}\subset V$, then the discrete associated weak variational is given by:

$$ \begin{cases}
\text{Find} \; u^n \in V= H_{0}^{1}(\Omega)\;\text{such that }   \\
 a(u_{h}^n,v_{h})= L(v_{h}), \; \text{for all} \; v_{h}\in V_{h} \\
  \end{cases}$$
Picard's iterative scheme implies solving a sequence of linear systems: \\
$$A[U^n]= F[U^{n-1}],$$
for a given initial guess $U^0$, where $A$ is the IGA-Galerkin stiffness matrix, $F$ is the right-hand side, and $U^n$ is a vector of control points.

\subsubsection{The Monge-Ampère equation}

We consider now the Monge-Ampère equation with non-homogeneous Dirichlet boundary conditions~\cite{R8}:

\begin{equation}
 \begin{cases}
\;\;\;\;\text{det}\;(\mathbf{H} (u)) & = f, \;\;\;\;\;\;\;  \text{in} \;\; \Omega, \\
\;\;\;\;  u &= g, \;\;\;\;\;\;\, \text{on} \;\;\; \Gamma=\partial \Omega,\\
\;\;\;\; u \, \text{is convex,}\\
\end{cases}
\label{eq2dMonge_Ampere}
\end{equation}
where $\text{det}\;(\mathbf{H} (u))$ denotes the determinant of the Hessian ($\mathbf{H} (u)$) of the function $u$. In this case, $\Omega \subset \mathbb{R}^{n}$ is a bounded domain with boundary $\Gamma=\partial \Omega$, and $f: \Omega \rightarrow \mathbb{R}$ is strictly positive and regular function.\\
If we define the operator $T$ by:
\[
\setlength\arraycolsep{0pt}
T\colon \begin{array}[t]{ >{\displaystyle}r >{{}}c<{{}}  >{\displaystyle}l } 
         \mathbb{H}^{2} (\Omega)  &\longrightarrow & \mathbb{H}^{2} (\Omega)\\ 
          u &\mapsto& T[u]= \Delta ^{-1} \left( (\Delta u)^{d} + d! (f - \text{det} (\mathbf{H} (u)) )   \right) ^{\frac{1}{d}},
         \end{array}
\]
then, Benamou Froese and Oberman~\cite{R8} prove that the solution $u$ of Monge–Ampère equation (\ref{eq2dMonge_Ampere}) is a fixed point of $T$.\\
We can rewrite (\ref{eq2dMonge_Ampere}) as follows:

\begin{equation}
 \begin{cases}
\;\;\;\; -\Delta u & = - \mathcal{G}  (u), \;\;\;\;\;\;\;  \text{in} \;\; \Omega, \\
\;\;\;\;  u &= g, \;\;\;\;\;\;\;\;\;\;\;\;\,\,\,\,\, \text{on} \;\;\; \Gamma=\partial \Omega,\\
\;\;\;\; u\, \text{is convex,}\\
\end{cases}
\label{eq2dMonge_Ampere_2}
\end{equation}
where $\mathcal{G}  (u) = \left( (\Delta u)^{d} + d! (f - \text{det} (\mathbf{H} (u)) )   \right) ^{\frac{1}{d}}. $ $d$ is the dimension space.\\     
Therefore, the solution is provided by the Picard iterative method as follows (Algorithm \ref{algo6}):

\begin{algorithm}
\caption{Picard Algorithm for Monge–Ampère equation}\label{algo6}
\begin{algorithmic}[1]
\Require $\text{Given an initial guess} \,u^{0}$ \text{and a fixed tolerance; tol}
\Ensure $u^{n_{\text{final}}}$
\State \text{Determine} $u ^{n}$ \text{as the solution to} 
\begin{equation}
 \begin{cases}
\;\;\;\; -\Delta u ^{n}  & = - \mathcal{G}  (u^{n-1}), \;\;\;\;\;\;\;  \text{in} \;\; \Omega, \\
\;\;\;\;  u ^{n} &= g, \;\;\;\;\;\;\;\;\;\;\;\;\;\;\; \,\,\,\,\,\,\,\, \, \text{on} \;\;\; \Gamma=\partial \Omega,\\
\;\;\;\; u ^{n}\, \text{is convex,}\\
\end{cases}
\label{eq2dMonge_Ampere_2}
\end{equation}
\State \text{Repeat until the} $L_2$ \text{norm of the residual of Picard satisfies:} 
\begin{equation*}
 \Vert u ^{n} - u ^{n-1}    \Vert _{L_{2}(\Omega)} \leq \text{tol}.
\end{equation*}
\end{algorithmic}
\end{algorithm}

\subsubsection{Fixed-point iterations}
The Picard sequence of the $2D$ Bratu problem (\ref{eq2dbratu}) is defined by:
$$u^{n+1}:= 	(-\Delta)^{-1} (f-\lambda e^{u^n} ),$$
and the Picard sequence of the $2D$ Monge–Ampère equation (\ref{eq2dMonge_Ampere}) is as follows:

$$u^{n+1}:= \Delta ^{-1} \left( (\Delta u^n)^{2} + 2 (f - \text{det} (\mathbf{H} (u^n)) )   \right) ^{\frac{1}{2}}.$$
The iterations $ (u^n )_{n}$, which are produced by the Picard iterative scheme, are used to construct the sequence $ (s_k )_{k}$ (\ref{fixed-point-extrapol}), where $s_k=u^k $, which is used in the extrapolation step to speed up the convergence of the Picard iterations.

\section{Vector extrapolation methods}\label{sec3}
In order to apply several vector extrapolation techniques to geometric multigrid approaches in the context of isogeometric analysis for nonlinear problem solutions, we will first go over these techniques in this section.
\subsection{Description of polynomial extrapolation methods}

Let $(s_k)_{k\in \mathbb{N}}$ be a sequence of vectors of $\mathbb{R}^{\mathbb{N}}$, produced by a fixed-point iteration, starting with a vector $s_0$.
\begin{equation}
    s_{k+1}= G(s_{k}),\; k=0,1,2,...
    \label{fixed-point-extrapol}
\end{equation}
which solution is denoted by $s$. We consider the transformation $T_k$ defined by:
\[
\setlength\arraycolsep{0pt}
T_k\colon \begin{array}[t]{ >{\displaystyle}r >{{}}c<{{}}  >{\displaystyle}l } 
         \mathbb{R}^{\mathbb{N}}  &\to& \mathbb{R}^{\mathbb{N}}\\ 
          s_k &\mapsto& t_{k,q},
         \end{array}
\]
with
$$t_{k,q}= \sum_{j=0}^{q} \gamma_{j}^{(q)}s_{k+j},$$
$t_{k,q}$ is the approximation produced by the reduced rank or by the minimal polynomial extrapolation methods of the limit or the anti-limit of $s_k$ as $k\rightarrow \infty$, where 
$$ \sum_{j=0}^{q} \gamma_{j}^{(q)}=1 \;\; \text{and}\; \sum_{j=0}^{q}\eta_{i,j} \gamma_{j}^{(q)}=0,\; i=0,...,q-1,$$
with the scalars $\eta_{i,j}$ defined by :
$$\eta_{i,j}= \begin{cases}
(\Delta^{2}s_{k+i},\Delta ^{2}s_{k+j}),& \text{for the RRE},   \\
 (\Delta s_{k+i},\Delta s_{k+j}),& \text{for the MPE.}  \\
  \end{cases}$$
$\Delta s_{k}$\, and $\Delta^{2}s_{k}$ denotes respectively the first and second forward differences of $s_k$ and are defined by:\\

\begin{equation*}
\begin{array}{rcl}
\Delta s_{k} & = & s_{k+1}- s{_k},\;\;\;\;\;\;\; \;\;k=0,1,...\\
\Delta ^ 2 s_{k}& = &  \Delta s_{k+1}-\Delta s_{k} ,\;\;\;k=0,1,...
\end{array}
\end{equation*}
Let us introduce the matrices:
$$ \Delta ^{i}S_{k,q}=[\Delta ^{i}s_{k},...,\Delta ^{i}s_{k+q-1}], \,i=1,2.$$
Using Schur’s formula, the approximation $t_{k,q}$ can be written in matrix form as 
$$t_{k,q}= s_{k}-\Delta S_{k,q}\left({ Y_{q}}^{T} \Delta^{2} S_{k,q} \right)^{-1} { Y_{q}}^{T}\Delta s_{k},$$
where 
$$Y_q= \begin{cases}
\Delta^2 S_{k,q},& \text{for the RRE},   \\
 \Delta S_{k,q},& \text{for the MPE.}  \\
  \end{cases}$$
In particular, for the RRE, the extrapolated approximation $t_{k,q}$ is given by:
$$t_{k,q}= s_{k}-\Delta S_{k,q}\Delta^{2} S_{k,q}^{+} \Delta s_{k},$$
where $\Delta^{2} S_{k,q}^{+}$ is the Moore-Penrose generalized inverse of $\Delta^{2} S_{k,q}$ defined by:
$$ \Delta^{2} S_{k,q}^{+}=\left(\Delta^{2} S_{k,q}^{T} \Delta^{2} S_{k,q} \right)^{-1} \Delta^{2} S_{k,q}^{T}. $$
Let $\tilde{T_k}$ be the new transformation from $T_k$ given by: 
\[
\setlength\arraycolsep{0pt}
\tilde{T_k}\colon \begin{array}[t]{ >{\displaystyle}r >{{}}c<{{}}  >{\displaystyle}l } 
         \mathbb{R}^{\mathbb{N}}  &\to& \mathbb{R}^{\mathbb{N}}\\ 
          s_k &\mapsto& \tilde{t}_{k,q},
         \end{array}
\]
with 
$$\tilde{t}_{k,q}= \sum_{j=0}^{q} \gamma_{j}^{(q)}s_{k+j+1},$$
$\tilde{t}_{k,q}$ is a new approximation of the limit or the anti-limit of $s_k, k\rightarrow \infty$.
Therefore, the generalized residual of $t_{k,q}$ has been defined as follows:

\begin{equation*}
\begin{array}{rcl}
\tilde{r}(t_{k,q}) & = & \tilde{t}_{k,q}- t_{k,q} \\
& = &  \Delta s_{k}-\Delta^{2} S_{k,q}\left({ Y_{q}}^{T} \Delta^{2} S_{k,q} \right)^{-1} { Y_{q}}^{T}\Delta s_{k}.
\end{array}
\end{equation*}

\subsection{Algorithms of polynomial extrapolation methods}
The reduced rank extrapolation (RRE) and the minimal polynomial extrapolation (MPE) methods \cite{R18,R30,R31} are explained in the following by Algorithm~\ref{algo2} and Algorithm~\ref{algo3}, respectively.
\begin{algorithm}
\caption{The RRE method}\label{algo2}
\begin{algorithmic}[1]
\Require $\text{Vectors} \,s_{k}, s_{k+1},...,s_{k+q+1}$
\Ensure $t_{k,q}^{\text{RRE}}: \text{the RRE extrapolated  approximation} $
\State $ \text{Compute}\, \Delta s_{i}= s_{i+1}-s_{i},\,\text{for} \, i=k,k+1,...,k+q$
\State $\text{Set } \Delta S_{q+1}=[\Delta s_{k},\Delta s_{k+1},...,\Delta s_{k+q}]$
\State $\text{Compute the QR factorization of}\, \Delta S_{q+1}, \text{namely}, \,\Delta S_{q+1}=Q_{q+1}R_{q+1}$
\State $ \text{Solve the linear system}\, R_{q+1}^{T}R_{q+1}d=e, \,\text{where} \,d=[d_0,...,d_q]^{T} \,\text{and}\, e=[1,...,1]^{T}$
\State $ \text{Set}\, \lambda=(\sum_{i=0}^{q}d_{i})^{-1}\, \text{and } \,\gamma_{j}= \lambda d_{i}, \,\text{for} \,i=0,...,q$
\State $ \text{Compute } \alpha=[\alpha_0,\alpha_1,...,\alpha_{q-1}]^{T}\, \text{where}\, \alpha_0=1-\gamma_0 \,\text{and } \,\alpha_{j}=\alpha_{j-1}-\gamma_{j} \,\text{for} \,j=1,...,q-1$
\State $ \text{Compute}\, t_{k,q}^{\text{RRE}}=s_k+ Q_{q}(R_{q}\alpha)$
\end{algorithmic}
\end{algorithm}

\begin{algorithm}
\caption{The MPE method}\label{algo3}
\begin{algorithmic}[1]
\Require $\text{Vectors} \,s_{k}, s_{k+1},...,s_{k+q+1}$
\Ensure $t_{k,q}^{\text{MPE}}: \text{the MPE extrapolated  approximation} $
\State $ \text{Compute}\, \Delta s_{i}= s_{i+1}-s_{i},\,\text{for} \, i=k,k+1,...,k+q$
\State $\text{Set } \Delta S_{q+1}=[\Delta s_{k},\Delta s_{k+1},...,\Delta s_{k+q}]$
\State $\text{Compute the QR factorization of}\, \Delta S_{q+1}, \text{namely}, \,\Delta S_{q+1}=Q_{q+1}R_{q+1}$
\State $ \text{Solve the upper triangular linear system}\, R_{q}d =-r_q, \,\text{where}\, r_{q}=[r_{0q},r_{1q},...,r_{q-1,q}]^{T} $
\State $ \text{Set}\,d_q=1 \,\text{and compute}\, \lambda=(\sum_{i=0}^{q}d_{i})^{-1}\, \text{and set} \,\gamma_{j}= \lambda d_{i}, \,\text{for} \,i=0,...,q$
\State $ \text{Compute } \alpha=[\alpha_0,\alpha_1,...,\alpha_{q-1}]^{T}\, \text{where}\, \alpha_0=1-\gamma_0 \,\text{and } \,\alpha_{j}=\alpha_{j-1}-\gamma_{j} \,\text{for} \,j=1,...,q-1$
\State $ \text{Compute}\, t_{k,q}^{\text{MPE}}=s_k+ Q_{q}(R_{q}\alpha)$
\end{algorithmic}
\end{algorithm}

\subsection{Quadratic convergence results}
We consider the linear or the nonlinear system of equations defined by:
\begin{equation}
\begin{array}{rcl}
F(x) & = & 0, \;\;\;\;\;\; F: \mathbb{R^N}\longrightarrow \mathbb{R^N},\label{eq3}
\end{array}
\end{equation} 
where $x^*$ is the exact solution of (\ref{eq3}). This equation can be expressed as 
\begin{equation}
\begin{array}{rcl}
G(x) & = & x, \;\;\;\;\;\; G: \mathbb{R^N}\longrightarrow \mathbb{R^N}.
\end{array}
\end{equation} 
Let $s_0$ be an initial vector of $\mathbb{R^N}$. the vector sequences $(s_k)_{k\in \mathbb{N}}$  of $\mathbb{R}^{\mathbb{N}}$, produced by the fixed point iteration, are given by:
\begin{equation}
s_{k+1} = G(s_{k}), k=0,1,2,... \label{eq5}
\end{equation} 
The algorithm which will be considered is the following:
\begin{itemize}
    \item Choose a starting point $x_0$,
    \item at the iteration $k$, we set $s_o = x_k $ and $s_{i+1} = G(s_i)$ for $i = 0,. . . , d_k$,
where $d_k$, is the degree of the minimal polynomial of $G’(x ^ *)$ for the vector $x_k -x^*$,
\item compute $x_{k+1}$ such that $x_{k+1} = t_{0,d_k}$, using one of the desired algorithms (\ref{algo2},\ref{algo3}).
\end{itemize}
Therefore, $x_{k+1}$ and the generalized residual $\tilde{r}(x_{k+1})$ of $x_{k+1}$ are given, respectively, by:
\begin{align}
x_{k+1}=t_{0,d_k}= s_{0}-\Delta S_{d_k}\left({ Y_{d_k}}^{T} \Delta^{2} S_{d_k} \right)^{-1} { Y_{d_k}}^{T}\Delta s_{0}, \label{eq6}
\end{align}
and
\begin{align}
   \tilde{r}(x_{k+1})  = \Delta s_{0}-\Delta^{2} S_{d_k}\left({ Y_{d_k}}^{T} \Delta^{2} S_{d_k} \right)^{-1} { Y_{d_k}}^{T}\Delta s_{0}.  \label{eq7}
\end{align}

We first provide a new estimation of the generalized residual norm for the RRE method in the following Theorem \ref{theom1}, which develops on an earlier finding in~\cite{R34}.

 \begin{theorem}\label{theom1}
Consider the $N \times (d_k + 1)$ matrix $\tilde{\Delta} S_{{d_k}+1} $, which is given by $\left [ \Delta s_0, \Delta^{2} S_{d_k}\right ]$. If  $\tilde{\Delta} S_{{d_k}+1}$ is of full rank, then
$$ {\Vert \tilde{r}(x_{k+1}) \Vert}^2_2 = \frac{1}{ e^{T}( \Delta S_{{d_k}+1}^T \Delta S_{{d_k}+1} )^{-1} e},$$ 
where $\Delta S_{{d_k}+1}= \left [ \Delta s_0, \Delta s_{1},..., \Delta s_{d_k} \right ],$ and $e=(1,\ldots,1)^{T}$ is a vector of ones in $\mathbb{R}^{d_k+1}$.
\end{theorem}
 
\begin{proof}
We have for the RRE method:
\begin{equation*}
\begin{array}{rcl}
\tilde{r}(x_{k+1})  
 & = &\Delta s_{0}-\Delta^{2} S_{d_k} \Delta^{2}S_{d_k} ^{+} \Delta s_{0} \\ 
 & = & (I- \Delta^{2} S_{d_k} \Delta^{2} S_{d_k}^{+} ) \Delta s_{0}.\\ 
\end{array}
\end{equation*}  
If we define the orthogonal projector $P_{d_k}$ by $P_{d_k}=I- \Delta^{2} S_{d_k} \Delta^{2} S_{d_k}^{+}$, we have on the one hand,\\ $\tilde{r}(x_{k+1})= P_{d_k}\Delta s_{0}.$ On the other hand, given that $P_{d_k}^2 = P_{d_k}$ and $P_{d_k}^T = P_{d_k}$, we obtain

\begin{equation}
\begin{array}{rcl}
\Vert \tilde{r}(x_{k+1}) \Vert^2_2 & = &  \langle  \; P_{d_k}\Delta s_{0} , \; P_{d_k}\Delta s_{0} \rangle \\
 & = & \langle  \; \Delta s_{0} , \; P_{d_k}\Delta s_{0} \rangle \\ 
 & = & \langle  \; \Delta s_{0} , \; \Delta s_{0} \rangle - \langle  \; \Delta s_{0} , \; \Delta^{2} S_{d_k} \Delta^{2} S_{d_k}^{+}\Delta s_{0} \rangle  \\ 
  & = & \langle  \; \Delta s_{0} , \; \Delta s_{0} \rangle - \langle  \; \Delta s_{0} , \; \Delta^{2} S_{d_k} \left({ \Delta^{2} S_{d_k}}^{T} \Delta^{2} S_{d_k} \right)^{-1} { \Delta^{2} S_{d_k}}^{T} \Delta s_{0} \rangle .   \label{eq11}\\
\end{array} 
\end{equation}  
We observe that the right-hand side of (\ref{eq11}) is a Schur complement for the matrix
\[
\tilde{ \Delta} S_{{d_k}+1} ^{T}\tilde{ \Delta} S_{{d_k}+1}= \left( \begin{array}{cc}
\Delta s_{0}^{T} \Delta s_{0} & \quad \Delta s_{0}^{T} \Delta^{2} S_{d_k} \\[1em]
\quad \Delta^{2} S_{d_k}^{T} \Delta s_{0} & \quad \Delta^{2} S_{d_k}^{T} \Delta^{2} S_{d_k} \\
\end{array} \right).
\]
Block UL factorization allows us to factor this matrix into a product of block upper and block lower triangular matrices:
\[
\tilde{ \Delta} S_{{d_k}+1} ^{T}\tilde{ \Delta} S_{{d_k}+1}= \left( \begin{array}{cc}
1 & \quad \Delta s_{0}^{T} \Delta^{2} S_{d_k} (\Delta^{2} S_{d_k}^{T} \Delta^{2} S_{d_k})^{-1}  \\[1em]
\quad 0 & \quad I \\
\end{array} \right) \left( \begin{array}{cc}
\Vert \tilde{r}(x_{k+1}) \Vert^2_2 & \quad 0  \\[1em]
\quad \Delta^{2} S_{d_k}^{T} \Delta s_{0} & \quad \Delta^{2} S_{d_k}^{T} \Delta^{2} S_{d_k} \\
\end{array} \right).  
\]
We denote by $\text{det}(M)$ the determinant of the matrix $M$. Since the matrices $\Delta^{2} S_{d_k}^{T} \Delta^{2} S_{d_k}$ and $\tilde{ \Delta} S_{{d_k}+1} ^{T}\tilde{ \Delta} S_{{d_k}+1}$ are nonsingular, by taking determinants on both sides, the following result holds 
$$ \Vert \tilde{r}(x_{k+1}) \Vert^2_2= \frac{\text{det} (\tilde{ \Delta} S_{{d_k}+1} ^{T}\tilde{ \Delta} S_{{d_k}+1})   }{\text{det}( \Delta^{2} S_{d_k}^{T} \Delta^{2} S_{d_k} )} =  \frac{1}{ e_1^{T}(\tilde{ \Delta} S_{{d_k}+1}^T \tilde{ \Delta} S_{{d_k}+1} )^{-1}e_1 }.$$
By definition:
\begin{equation*}
\begin{array}{rcl}
\tilde{ \Delta} S_{{d_k}+1}  & = &\left [ \Delta s_0, \Delta^{2} S_{d_k}\right ]\\
 & = &\left [ \Delta s_0, \Delta ^ 2 s_{0},\Delta ^ 2 s_{1},...,\Delta ^ 2 s_{d_k-1}\right ] \\ 
 & = & \left [ \Delta s_0, \Delta s_{1}- \Delta s_{0}, \Delta s_{2}- \Delta s_{1},..., \Delta s_{d_k}-\Delta s_{d_k-1}    \right ]\\ 
& = & \Delta S_{{d_k}+1} W,
\end{array}
\end{equation*}
where the $(d_k + 1) \times (d_k + 1)$ matrix $W$ is defined by 
\[
W = \begin{bmatrix}
1 & -1 & 0 & 0 & \cdots & 0 & 0 \\
0 & 1 & -1 & 0 & \cdots & 0 & 0 \\
0 & 0 & 1 & -1 & \cdots & 0 & 0 \\
0 & 0 & 0 & 1 & \cdots & 0 & 0 \\
\vdots & \vdots & \vdots & \vdots & \ddots & \vdots & \vdots \\
0 & 0 & 0 & 0 & \cdots & 1 & -1 \\
0 & 0 & 0 & 0 & \cdots & 0 & 1 
\end{bmatrix},
\]
Since $W$ is an upper triangular matrix with all diagonal elements equal to $1$, hence $W$ is nonsingular. Moreover, $W^{-1}$ have the form:
%
\[ W^{-1} = \begin{bmatrix}
1 & 1 & 1 & 1 & \cdots & 1 & 1 \\
0 & 1 & 1 & 1 & \cdots & 1 & 1 \\
0 & 0 & 1 & 1 & \cdots & 1 & 1 \\
0 & 0 & 0 & 1 & \cdots & 1 & 1 \\
\vdots & \vdots & \vdots & \vdots & \ddots & \vdots & \vdots \\
0 & 0 & 0 & 0 & \cdots & 1 & 1 \\
0 & 0 & 0 & 0 & \cdots & 0 & 1 
\end{bmatrix}.\]
From $ \tilde{ \Delta} S_{{d_k}+1} = \Delta S_{{d_k}+1} W,$ we have:
$$ e_1^{T}(\tilde{ \Delta} S_{{d_k}+1}^T \tilde{ \Delta} S_{{d_k}+1} )^{-1}e_1 = e_1^{T}  W^{-1}  (\Delta S_{{d_k}+1}^{T} \Delta S_{{d_k}+1} )^{-1}  (W^{T})^{-1} e_1.$$
Since, $e_1^T W^{-1}=e^T$ and $(W^{T})^{-1} e_1=e,$ \\
we conclude that 
$$ {\Vert \tilde{r}(x_{k+1}) \Vert}^2_2 = \frac{1}{ e_1^{T}(\tilde{ \Delta} S_{{d_k}+1}^T \tilde{ \Delta} S_{{d_k}+1} )^{-1}e_1 } = \frac{1}{e^{T} (\Delta S_{{d_k}+1}^{T} \Delta S_{{d_k}+1} )^{-1} e}.$$ 
\end{proof}

\begin{remark}
In practice, the new approximation of the generalized residual for the RRE method in Theorem \ref{theom1} plays a very crucial role, as it indicates that we are able to recover the 2-norm of the generalized residual before computing the extrapolated approximation in the Algorithm \ref{algo2} (step $7$) which will allow us to reduce the computational cost of the algorithm; assuming that the matrix $\Delta S_{{d_k}+1}$ is of full rank. Then we can determine a QR factorization $\Delta S_{d_k+1}= Q_{d_k+1} R_{d_k+1}$, where $Q_{d_k+1} = [q_0,q_1,... ,q_{d_k}] \in \mathbb{R}^{N\times (d_k+1)}$
 has orthonormal columns, and $R_{d_k+1} \in \mathbb{R}^{(d_k+1)\times (d_k+1)}$  is upper triangular with positive diagonal
entries. Hence, the norm of the generalized residual can be expressed as
$$ {\Vert \tilde{r}(x_{k+1}) \Vert}^2_2 = \frac{1}{ e^{T}( R_{d_k+1}^T R_{d_k+1} )^{-1} e}.$$ 
Since in steps $4$ and $5$ of Algorithm \ref{algo2} we need to solve the linear system $R_{d_k+1}^{T}R_{d_k+1}d=e$, where $d=[d_0,...,d_{d_k}]^{T}$, and if we set $\lambda=(\sum_{i=0}^{d_k}d_{i})^{-1}$, we obtain that 
$$ {\Vert \tilde{r}(x_{k+1}) \Vert}^2_2 = \frac{1}{ e^{T} d}= \lambda.$$
\end{remark}
$$$$

Skelboe \cite{R35} provided arguments for the quadratic convergence of ${(x_k)}_k$ to $x^*$ for the MPE, whereas Beuneu \cite{R36} provided arguments for a certain class of extrapolation techniques. Smith, Ford, and Sidi \cite{R37} have identified a weakness in Skelboe's evidence. Sadok and Jbilou have discovered an identical hole in Beuneu's proof. Moreover, in \cite{R38}, the authors give a sufficient condition, complete, and satisfactory proof of the quadratic convergence.

 

$$$$
We set $G’( x^*) = J $ and assume that $G$ satisfies the following conditions:

\begin{itemize}
    \item [i)] The matrix $J-I$ is regular. We set $M= \Vert(J-I)^{-1}\Vert$.

    \item[ii)]  The Frechet derivative $G’$ of $G $ satisfies the Lipschitz condition 
    $$ 
    \Vert G'(x)-G'(y) \Vert \leq L \Vert x-y \Vert ,\;\;\; \forall x,y \in D,
    $$
where $D$ is an open and convex subset of $\mathbb{C}^ p$.
\end{itemize}

\begin{theorem}[\cite{R38}]\label{theom2}
If $G$ satisfies i), and ii). Moreover, if 
$$\exists \alpha > 0, \exists K,\;\; \alpha_k(x_k) > \alpha, \forall k\geq K,$$
\text{then there exists a neighbourhood} $U$ \text{of} $x ^*$ \text{such that} $\forall x_0 \in U $

$$ \Vert x_{k+1} -x^* \Vert = O(\Vert x_k - x^*\Vert^2),\; \;\;\;\text{for RRE and MPE.}$$

 \end{theorem} 
 
 \begin{proof}
See, for instance, \cite{R38}. 
 \end{proof}
 
Our objective is to demonstrate the quadratic convergence without the need for hypotheses on $\alpha_k$~\cite{R38}. Then, in the following Theorem \ref{theom3}, we prove a new theoretical result on the relation between the residual norm and the error norm for the RRE method.\\
We first give the following lemma:

\begin{lem}\label{minimpoly}
Let  $M(\lambda)= \sum_{i=0}^{d_k} \eta_i^{(k)} \lambda ^{i},\,\eta_{d_k}^{(k)}=1, $ be the minimal polynomial of the matrix $J$ with respect to the vector $(x_k - x^*)$, therefore, 
\[
\sum_{i=0}^{d_k} \vert \eta_i^{(k)}\vert \leq M_0.
\]
\end{lem}

\begin{proof}
The minimal polynomial $M(\lambda)$ of $J$ satisfies 
$$M(J) (x_k - x^*):=  \sum_{i=0}^{d_k} \eta_i^{(k)} J ^{i}(x_k - x^*) = 0,$$
where
$$ d_k = \min \{p/ \sum_{i=0}^{p} \eta_i^{(k)} J ^{i}(x_k - x^*) =0,\, \eta_{p}^{(k)}=1\} .$$    
Moreover, the characteristic polynomial $P_c(\lambda)$ of $J$ is defined as:
$$P_{c }(\lambda)= \text{det} (\lambda I- J)= \prod _{i=1} ^ {N} (\lambda - \lambda_i)^{\alpha_i}.$$
$M(\lambda)$ divides $P_c(\lambda)$. This implies that the roots (eigenvalues) of $M(\lambda)$ are a subset of the roots of $P_c(\lambda)$. Let $\lambda_1,\lambda_2,...,\lambda_N$ be the eigenvalues of $J$. The minimal polynomial $M(\lambda)$ is:
$$ M(\lambda)= \prod _{i=1} ^ {d_k} (\lambda - \lambda_i)^{\beta_i},$$
where $1 \leq \beta_i \leq \alpha_i,$
and $\lambda_1,\lambda_2,...,\lambda_{d_k}$ are a subset of the eigenvalues of $J$.
The coefficients $\eta_i^{(k)}$ of the minimal polynomial are related to the elementary symmetric polynomials of the eigenvalues $\lambda_1,\lambda_2,...,\lambda_{d_k}$. These coefficients can be expressed as:
$$ \eta_i^{(k)} = (-1)^{{d_k}-i} e_{{d_k}-i}(\lambda_1,\lambda_2,...,\lambda_{d_k}),$$
where $e_{i}$ is the i-th elementary symmetric polynomial that satisfies:
$$
\begin{cases}
e_i(\lambda_1, \lambda_2, \ldots, \lambda_{d_k}) = \sum_{1 \leq j_1 < j_2 < \cdots < j_i \leq d_k} \lambda_{j_1} \lambda_{j_2} \cdots \lambda_{j_i}, \\
 \vert e_{i}(\lambda_1,\lambda_2,...,\lambda_{d_k}) \vert \leq   \binom{d_k}{i}  \alpha^{i}.             \\
  \end{cases}$$
Therefore,
$$\vert \eta_i^{(k)} \vert \leq   \binom{d_k}{d_k -i}  \alpha^{d_k -i}=\binom{d_k}{i}  \alpha^{d_k -i} ,$$ 
it follows that 
$$
\sum_{i=0}^{d_k} \vert \eta_i^{(k)}\vert \leq \sum_{i=0}^{d_k}  \binom{d_k}{i}  \alpha^{d_k -i}    ,
$$
by the binomial theorem, we get:
$$
\sum_{i=0}^{d_k} \vert \eta_i^{(k)}\vert \leq (1 + \alpha )^{d_k} \leq M_0,
$$
where $M_0=(1 + \alpha )^{N}.$

\end{proof}

 \begin{theorem}\label{theom3}
If $G$ satisfies i), and ii), then $\forall x_0\in D $,
$$\Vert \tilde{r}(x_{k+1} )\Vert = O(\Vert x_k -x^*\Vert^2) .$$
 \end{theorem}

 \begin{proof}
Consider the functions $F(x) = G(x) - x$ and $g(x) = F(x) - F ^{'}(x^*)(x - x^*)$. This implies:

\[
\|g(x)\| \leq \frac{1}{2} L \|x - x^*\|^2,
\]
for all $x$ in the domain $D$ (refer to \cite{R39}).\\
We have $\Delta s_i = s_{i+1}- s_i= G(s_i) - G(x^*) - (s_i - x^*)$,
thus
\[
\Delta s_i = (J - I)(s_i - x^*) + g(s_i) = (J - I)(s_i - x^*) + \varepsilon_{i,1}(s_i - x^*),
\]
with $\Vert \varepsilon_{i,1}(s_i - x^*)\Vert = O(\Vert s_i - x^*\Vert^2)$, and
$$
s_{i+1} - x^* = \Delta s_i + s_i - x^* = (J - I)(s_i - x^*) + \varepsilon_{i,1}(s_i - x^*) + ( s_i - x^*) = J(s_i - x^*) + \epsilon_{i,1}(s_i - x^*).
$$
Using these equations, it easily follows that
\begin{align*}
\Delta^2 s_i &= \Delta s_{i+1} - \Delta s_i \\
&= (J - I)(s_{i+1} - x^*) + \varepsilon_{i,1}(s_{i+1} - x^*) - (J - I)(s_i - x^*) - \varepsilon_{i,1}(s_i - x^*) \\
&= (J - I)^2 J ^{i} (x_k - x^*) + \varepsilon_{i,2}(x_k - x^*),
\end{align*}
with $\Vert \varepsilon_{i,2}(x_k - x^*)\Vert = O(\Vert x_k - x^*\Vert^2)$.\\
Now, let try to write $\Delta s_0= \Delta^{2} S_{d_k} a^{(k)}+ \varepsilon(x_k - x^*)$, where $\Vert \varepsilon(x_k - x^*)\Vert = O(\Vert x_k - x^*\Vert ^2).$\\
Next, let us multiply each $\Delta^2 s_i$ by $\eta_i^{(k)}$ and then sum up; thus from the expression of $\Delta^2 s_i$ we obtain:
\begin{align*}
\sum_{i=0}^{d_k} \eta_i^{(k)} \Delta^2 s_i&= (J - I)^2\sum_{i=0}^{d_k} \eta_i^{(k)} J^i (x_k - x^*) + \sum_{i=0}^{d_k} \eta_i^{(k)} \varepsilon_{i,2} (x_k - x^*) \\
&= \sum_{i=0}^{d_k} \eta_i^{(k)} \varepsilon_{i,2} (x_k - x^*). 
\end{align*}
Since $\sum_{i=0}^{d_k} \vert \eta_i^{(k)}\vert \leq M_0$ (by lemma \ref{minimpoly}), we have on the one hand
\[
\sum_{i=0}^{d_k} \eta_i^{(k)} \Delta^2 s_i = \varepsilon_{3 }(x_k - x^*), \text{ with } \Vert \varepsilon_{3} (x_k - x^*)\Vert = O(\Vert x_k - x^*\Vert^2),
\]
and
\[
\sum_{i=0}^{d_k} \eta_i ^{(k)} J^i \Delta^2 s_0 = (J - I)\varepsilon_{4} (x_k - x^*), 
\]
with  $\Vert \varepsilon_{4} (x_k - x^*)\Vert = O(\Vert x_k - x^*\Vert ^2)$.\\
On the other hand,
\[
\sum_{i=0}^{d_k} \eta_i ^{(k)}  (J^i - I)\Delta^2 s_0 = -\sum_{i=0}^{d_k} \eta_i^{(k)}  \Delta^2 s_0 + (J - I)\varepsilon_{4 }(x_k - x^*).
\]
We are going to prove by contradiction that $\sum_{i=0}^{d_k} \eta_i ^{(k)} \not=0.$ Suppose that $\sum_{i=0}^{d_k} \eta_i ^{(k)} =0,$ which means that $M(1)=0$, and since the minimal polynomial $M(\lambda)$ divides every annihilating polynomial, including the characteristic polynomial $P_c$, which is annihilating by the Cayley-Hamilton theorem. Then, $1$ is an eigenvalue of $J$, which implies that $0$ is an eigenvalue of the matrix $J-I$. This is impossible because $J-I$ is nonsingular.
Thus, by using the above equation, we get:
$$\frac{\eta_1^{(k)}}{\sum_{i=0}^{d_k} \eta_i^{(k)}} \Delta^2 s_0 + \sum_{i=2}^{d_k}  \frac{\eta_i^{(k)}}{\sum_{i=0}^{d_k} \eta_i^{(k)}} ( J^{i-1}+...+ I ) \Delta^2 s_0= -(J-I)^{-1} \Delta^2 s_0 + \varepsilon_{4}(x_k - x^*).$$
Then, we have
\begin{equation*}
\begin{array}{rcl}
(J - I)^{-1} \Delta^2 s_0 & = & \sum_{i=0}^{d_k - 1} a_i ^{(k)} J^ {i} \Delta^2 s_0 + \varepsilon_4 (x_k - x^*) \\
& = & \sum_{i=0}^{d_k - 1} a_i ^{(k)}\Delta^2 s_i + \varepsilon_4 (x_k - x^*).\\
\end{array}
\end{equation*} 
Moreover,
\begin{equation*}
\begin{array}{rcl}
\Delta s_0 & = &(J - I)^{-1} \Delta^2 s_0 + \varepsilon_5 (x_k - x^*), \text{ with } \Vert \varepsilon_5 (x_k - x^*)\Vert = O(\Vert x_k - x^*\Vert^2), \\
& = &  \Delta^{2} S_{d_k} a^{(k)} + \varepsilon(x_k - x^*), \text{ with } \varepsilon(x) = \varepsilon_4 (x) + \varepsilon_5 (x).\\
\end{array}
\end{equation*} 
Therefore, we can rewrite the generalized residual of $x_{k+1}$ for the RRE method as :
\begin{equation*}
\begin{array}{rcl}
\tilde{r}(x_{k+1}) & = & \Delta s_{0}-\Delta^{2} S_{d_k}\Delta^{2} S_{d_k}^{+}\Delta s_{0} \\
& = &   \Delta^{2} S_{d_k} a^{(k)} + \varepsilon(x_k - x^*) - \Delta^{2} S_{d_k}\Delta^{2} S_{d_k}^{+} \left(\Delta^{2} S_{d_k} a^{(k)} + \varepsilon(x_k - x^*)\right) \\
& = &   \left (  I- \Delta^{2} S_{d_k} \Delta^{2} S_{d_k}^{+}    \right) \varepsilon(x_k - x^*). \\
\end{array}
\end{equation*} 
Finally, since we have $\Vert I- \Delta^{2} S_{d_k} \Delta^{2} S_{d_k}^{+}\Vert\leq 1,$ \\ 
as a result, we get 
$$\Vert \tilde{r}(x_{k+1}) \Vert  \leq \Vert \varepsilon(x_k - x^*) \Vert.$$
Consequently, 
$$\Vert \tilde{r}(x_{k+1}) \Vert  =  O(\Vert (x_k - x^*)^2 \Vert).$$
\end{proof}

\newpage
\subsection{The Anderson acceleration method}
The following includes a description and the implementation of the Anderson acceleration method.\\
Assume that the vector sequences $(s_k)_{k\in \mathbb{N}}$  of $\mathbb{R}^{\mathbb{N}}$ are produced by the fixed point iteration (\ref{eq5}). For the purpose of accelerating this fixed-point iteration, the standard general form of Anderson acceleration ~\cite{R22,R23} is as follows (Algorithm \ref{algo4}):
\begin{algorithm}
\caption{The Anderson acceleration (AA)}\label{algo4}
\begin{algorithmic}[1]
\Require $\text{Given an initial guess} \,x_{0}$ \text{and} $m\geq 1$
\Ensure $x_{k+1}: \text{the Anderson extrapolated  approximation} $
\State \text{Set} $s_0= x_0$ \text{and} $s_{1}=G(s_0)$
\For{$k=1,2...$}
            \State \text{Set} $m_k= \min \{m,k\}$
            \State \text{Compute} $F_k= (f_{k-m_k},...,f_k ),$ \text{where} $f_i= G(s_i)- s_i,$ \,\, $i=k-m_k,...,k$   
            \State \text{Determine} $\beta ^ {(k)}= (\beta_0 ^ {(k)},...,\beta_{m_k} ^ {(k)})^T$ by solving 
           	\begin{equation}
			\begin{aligned}
				\min_{\beta= (\beta_0, \ldots, \beta_{m_k})^T } & \quad \Vert F_k \beta \Vert_2 \\
				\text{s.t.} & \quad \sum_{i=0}^{m_k} \beta_i = 1
			\end{aligned}
			\label{eq6}
		\end{equation}
            \State \text{Compute} $x_{k+1}= \sum_{i=0}^{m_k} \beta_{i}^{(k)} G(s_{k-m_k +i})$
\EndFor
\end{algorithmic}
\end{algorithm}

For the sake of explanation, the Anderson depth in Algorithm \ref{algo4} is denoted by $m_k$. The parameter $m$ sets the maximum depth, and in the following, we commonly represent the algorithm by AA($m$).

It should be noted that a more general version,
\begin{equation}
\begin{array}{rcl}
x_{k+1} & = & (1-\alpha_k)\sum_{i=0}^{m_k} \beta_{i}^ {(k)} s_{k-m_k +i} + \alpha_k \sum_{i=0} ^{m_k} \beta_{i}^{(k)} G(s_{k-m_k +i}),\label{eq7}
\end{array}
\end{equation} 
is possible with the original Anderson acceleration approach, where $\alpha_k > 0$ is a relaxation parameter. The Anderson acceleration (AA) method is sometimes referred to as the Anderson mixing method in some application areas, such as electronic-structure computations, where $\alpha_k$ is known as the Anderson mixing coefficient. Here, it is appropriate to just take into account $\alpha_k = 1$, which provides the step in Algorithm \ref{algo4}.

There are various approaches to solving the constrained linear least-squares issue (\ref{eq6}) in Algorithm \ref{algo4}; see \cite{R26} for a few of them. Our choice is to redefine it in the unconstrained manner that \cite{R23,R25} suggests.
Therefore, we define 
$$ \Delta f_i = f_{i+1}- f_i , \,\,\, i=k-m_k,..., k-1,$$
and 
$$   \mathcal{F}_k= (\Delta f_{k-m_k},...,\Delta f_{k-1}).  $$
Hence, the least-squares issue (\ref{eq6}) is equivalent to

\begin{equation}
   \displaystyle \min_{\theta= (\theta_0,...,\theta_{m_k-1})^T } \Vert  f_k - \mathcal{F}_k \theta) \Vert_2,    
   \label{eq8}
\end{equation}
in which $\beta$ and $\theta$ are connected by

$$
 \begin{cases}
\beta_0= & = \theta_0, \\
\beta_i &= \theta_i - \theta_{i-1}, \;\;\; 1\leq i \leq {m_k -1},\\
 \beta_{m_k}&= 1- \theta_{m_k-1},\\
\end{cases}
$$
and 
$$  \theta_i= \sum_{j=0}^{i}  \beta_j, \;\;\; i=0,...,m_k -1.            $$
A modified variant of Anderson acceleration results from this unconstrained least-squares issue (\ref{eq8}). Using $\theta^{(k)}= (\theta ^{(k)}_0,...,\theta^{(k)}_{m_k-1})^T $ to indicate the least-squares solution, then we have

$$  x_{k+1}= G(s_k) - \sum_{i=0}^{m_k -1}  \theta_i^{(k)}  [G(s_{k- m_k +i +1}) -  G(s_{k- m_k +i})]  =    G(s_k) - \mathcal{G}_k \theta ^ {(k)},                                    $$
where 
$$   \mathcal{G}_k= (\Delta G_{k-m_k},...,\Delta G_{k-1}),     $$
with
$$    \Delta G_i = G(s_{i+1})- G(s_i) , \,\,\, i=k-m_k,..., k-1.$$
Next, Anderson's acceleration turns into (Algorithm \ref{algo5})

\begin{algorithm}
\caption{The Anderson acceleration (AA)}\label{algo5}
\begin{algorithmic}[1]
\Require $\text{Given an initial guess} \,x_{0}$ \text{and} $m\geq 1$
\Ensure $x_{k+1}: \text{the Anderson extrapolated  approximation} $

\State \text{Set} $s_0= x_0$ \text{and} $s_{1}=G(s_0)$
\For{$k=1,2...$}
            \State \text{Set} $m_k= \min \{m,k\}$
            \State \text{Calculate}  $G(s_k)$ \text{and let } $f_k = G(s_k)- s_k$  
            \State \text{Determine} $\theta ^ {(k)}= (\theta_0 ^ {(k)},...,\theta_{m_k-1} ^ {(k)})^T$ by solving 
            \begin{equation}
   	\min_{\theta= (\theta_0,...,\theta_{m_k-1})^T } \Vert f_k - \mathcal{F}_k \theta \Vert_2
	\label{eq10}
        \end{equation}
            \State \text{Compute} $x_{k+1}=  G(s_k) - \mathcal{G}_k \theta ^ {(k)}$
\EndFor

\end{algorithmic}
\end{algorithm}

\section{Numerical experiments}\label{sec4}

In this section, we are interested in demonstrating numerically the effectiveness of vector extrapolation methods when integrated with geometric multigrid methods with IGA for solving nonlinear problems.

\subsection{Numerical study of the nonlinear one-dimensional Bratu's equation}\label{sub1sec4}

We start by testing the Picard iterative method
on the following one-dimensional Bratu’s nonlinear boundary value problem (\ref{eq1dbratu}) given
by:
 \begin{equation} \label{eq1dbratu}
	\begin{cases}
		-\frac{\partial^2 u(x)}{\partial x^2} + \lambda e^{u(x)} & = f(x), \;\;\;\;\;\;\; 0<x<1, \\
		\;\;\;\;\; u(0)=u(1) & = 0. \\
	\end{cases}
\end{equation}
The exact solution and the right-hand side are given by:
$$u_(x)= \sin(2k\pi x),$$
$$f(x) = (2k\pi)^{2}\sin(2k\pi x) +\lambda e^{\sin(2k\pi x)}.$$ 
In our approach, we will simply use a single V-cycle to solve the linear system at each Picard iteration. We will use the V-Cycle (1,1) scheme with weighted Jacobi as a relaxation method where the relaxation parameter is $\omega = 2/3$. The resulting method is noted as Picard-MG. The following Figure~\ref{figure1} shows the convergence results of the Picard-MG method for different values of $\lambda$; in a fixed grid of $N=64$ points and for each fixed value of spline degree $p$, we report the number of iterations to reduce the “Relative Residual" which is the $L_2$ norm of 
the difference between successive Picard iterations
$\frac{\Vert U^{n}-U^{n-1} \Vert_{L_{2}}}{\Vert U^{n} \Vert_{L_2}}$, where $U^{n-1}, U^{n}$ are the approximations obtained by the Picard method at iterations $n-1$ and $n$, respectively. The stopping criteria is when the accuracy of $10^{-12}$ is reached.

 \begin{figure}[ht!]
    \centering
    \includegraphics[width=4cm,height=4cm]{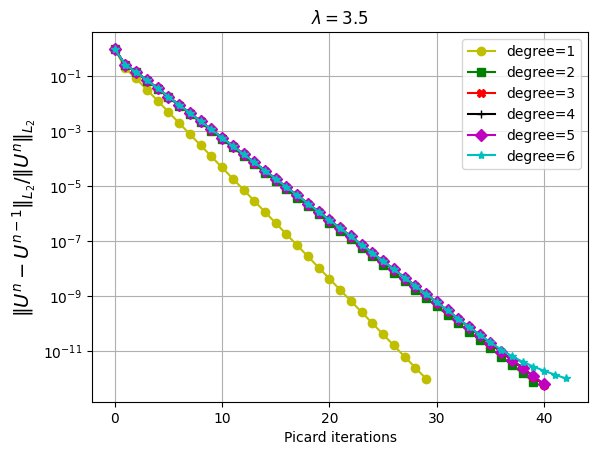}
    \includegraphics[width=4cm,height=4cm]{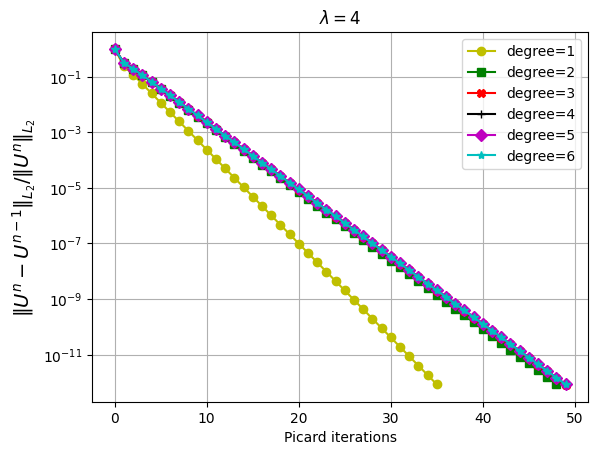}
    \includegraphics[width=4cm,height=4cm]{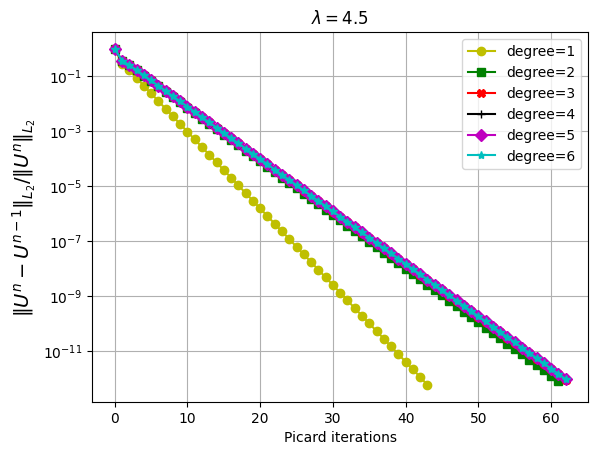}
    \includegraphics[width=4cm,height=4cm]{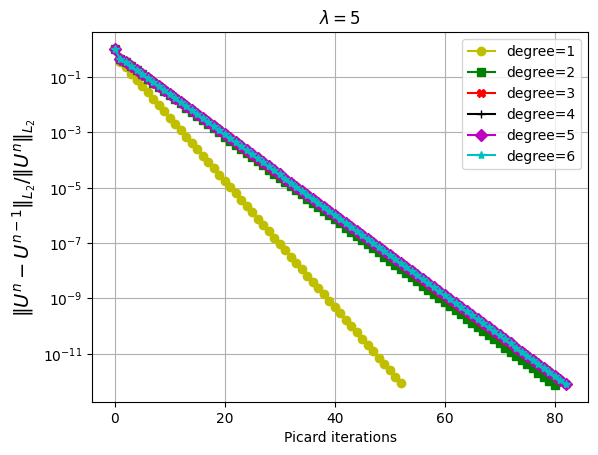}
    \caption{Convergence results of Picard-MG method for different values of spline degree $p$ and  $\lambda$}
    \label{figure1}
\end{figure}

Figure~\ref{figure1} shows that the Picard-MG method converges quickly for spline degree $1$ and its convergence is relatively independent of the degree of spline $p$. This is the same convergence behavior for different values of the parameter $\lambda$ in the Bratu problem. 

Moreover, in Figure~\ref{figure2}, numerical results for the Picard-MG method and the polynomial extrapolation methods coupled with Picard iterations (MPE-Picard-MG and RRE-Picard-MG) are given. For the purpose of comparison, results for the Anderson accelerated Picard-MG method are also given (AA-Picard-MG). Firstly, we note that the Picard-MG technique converges for small values of the parameter $\lambda$. However, we get superior convergence results in terms of the number of iterations and the $L_2$ error norm when we apply the restarted version of polynomial extrapolation methods to the Picard iterations than using the Picard approach alone. However, the figure clearly shows that the d'Anderson approaches are not as effective as the MPE and RRE methods.

\begin{figure}[ht!]
    \centering
    \includegraphics[width=6.5cm,height=6.5cm]{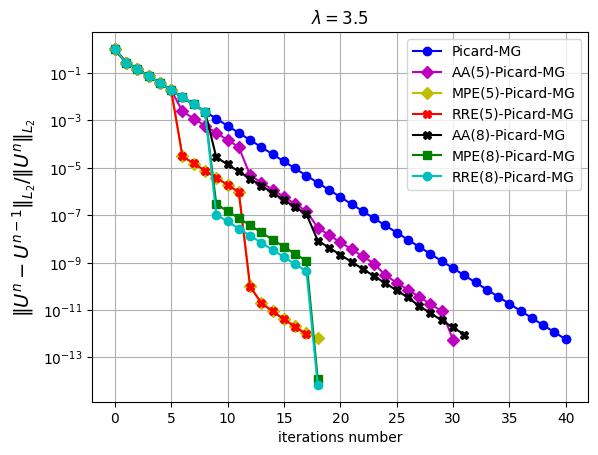}
    \includegraphics[width=6.5cm,height=6.5cm]{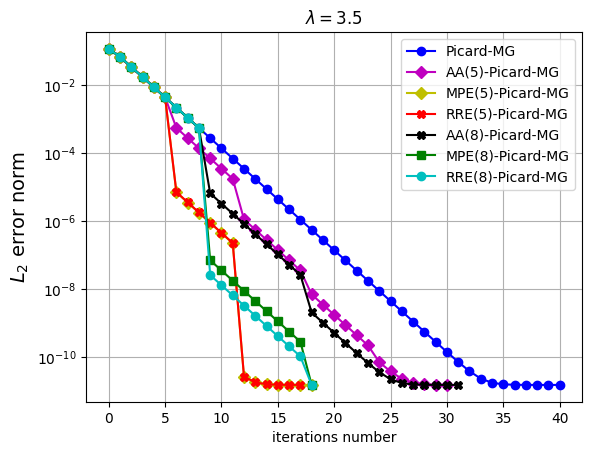}
  \caption{ Convergence results of the Picard method with and without extrapolation acceleration techniques for $\lambda=3.5$}
   \label{figure2}
\end{figure}
In the following, we will explain how the parameter $\lambda$ affects the computation of Picard's convergence. To achieve this, we will perform a numerical analysis of the effectiveness of our technique, which involves accelerating Picard's iterations through the use of restarted polynomial extrapolation methods \cite{R14}. We note that for high values of $\lambda$, the Picard-MG method diverges (see Table~\ref{table1}). Therefore, our goal is to develop a more precise and effective method for resolving Bratu's problem. We propose to use the restarted minimal polynomial and the reduced rank extrapolation methods (MPE-Picard-MG and RRE-Picard-MG) described in Algorithms~\ref{algo2} and~\ref{algo3}; more precisely, after $q$ iterations of the Picard-MG method, we apply MPE and RRE in its restarted version for the acceleration of the convergence of the Picard iterative method. We illustrate the performance of this new numerical method in Table~\ref{table1} which shows the convergence results of the Picard iterative method with one V-cycle scheme as a linear solver (Picard-MG), the restarted MPE and RRE methods applied to Picard iterations, and the Anderson-accelerated Picard-MG method. In order to show the limitation of the Picard-MG method, we also compare it with the Picard-slu method, which uses the LU factorization of the stiffness matrix to solve the linear system. Thus, we fix a high value of $\lambda=7$, we choose the spline degree $p=5$, and we vary the number of points $N$ (or size grid $h$). For each size grid, we report the number of iterations and CPU time in seconds to achieve a tolerance of $10^{-12}$. We also report the $L_2$ error norm, which is the norm of the difference between the exact and approached solutions of the Bratu equation.

\begin{table}[ht]
    \centering  
      \caption{Extrapolation and Picard methods were applied to the Bratu equation at different grid sizes and for spline degree $p=5$ and $\lambda=7$}
      \label{table16}
\begin{tabular}{ |c|c|c|c|c|c| } 
\hline
Grid size $h$ & Method &iter  & Relative Residual& $L_{2}\text{-err}$  & CPU(s)\\
\hline
\multirow{3}{4em}{1/8} 
&Picard-slu & $1000^{a}$& 3.86e-01 & 1.40e-01 & 5.66\\ 
&Picard-MG & $1000^{a}$& 3.73e-01 &1.35e-01 & 31.81\\ 
& AA(5)-Picard-MG & 55 & 2.70e-13 &5.68e-06 &1.36\\  
& RRE(5)-Picard-MG & 25 & 2.68e-13& 5.68e-06 & 0.72\\  
& MPE(5)-Picard-MG & 25 &2.63e-13  &5.68e-06 & 0.69\\  
& RRE(8)-Picard-MG &37  & 2.79e-16 &5.68e-06 &  1.13\\  
& MPE(8)-Picard-MG &37  & 2.16e-16 &5.68e-06 & 1.12 \\  
\hline
\multirow{3}{4em}{1/16} 
&Picard-slu & $1000^{a}$& 3.86e-01 & 1.40e-01 & 37.22\\ 
&Picard-MG & $1000^{a}$& 3.66e-01 & 1.32e-01 & 64.07\\ 
& AA(5)-Picard-MG & 57 & 9.60e-13& 6.76e-08 &6.52\\  
& RRE(5)-Picard-MG & 25 &4.73e-13 &  6.76e-08 & 2.78\\  
& MPE(5)-Picard-MG & 25 & 5.19e-13&  6.76e-08 & 2.60\\  
& RRE(8)-Picard-MG &37   &4.90e-16   &6.76e-08 & 5.89 \\  
& MPE(8)-Picard-MG &37  & 1.07e-15& 6.76e-08 & 5.54 \\  
\hline

\multirow{3}{4em}{1/32} 
&Picard-slu & $1000^{a}$ &3.86e-01 & 1.40e-01 &75.76 \\ 
&Picard-MG &$1000^{a}$& 3.79e-01& 1.37e-01 &122.90\\ 
& AA(5)-Picard-MG & 57 & 5.28e-13&  9.64e-10 & 12.44 \\  
& RRE(5)-Picard-MG  & 25& 7.35e-14&  9.64e-10 & 6.18 \\ 
& MPE(5)-Picard-MG  & 25& 7.72e-14 & 9.64e-10 & 6.09\\ 
& RRE(8)-Picard-MG &37  & 1.35e-15 &9.64e-10& 7.01 \\  
& MPE(8)-Picard-MG & 37 &5.58e-16 &  9.64e-10& 7.99 \\  
\hline 
\multirow{3}{4em}{1/64} 
&Picard-slu & $1000^{a}$& 3.86e-01&  1.40e-01 & 98.54\\ 
&Picard-MG & $1000^{a}$ &3.86e-01 & 1.40e-01 &223.09 \\ 
& AA(5)-Picard-MG & 56 &9.26e-13 &  1.46e-11 & 18.28 \\  
& RRE(5)-Picard-MG   & 25& 8.14e-14& 1.46e-11& 12.91\\  
& MPE(5)-Picard-MG   & 25& 8.61e-14& 1.46e-11& 10.53 \\  
& RRE(8)-Picard-MG & 37 &7.69e-16  & 1.46e-11& 16,23 \\  
& MPE(8)-Picard-MG & 37 & 7.43e-16& 1.46e-11 & 15.64 \\  
\hline 
\multirow{3}{4em}{1/128} 
&Picard-slu & $1000^{a}$ &3.86e-01 & 1.40e-01 &160.68 \\ 
&Picard-MG & 592& 9.53e-13& 3.98e-13 & 301.49\\ 
& AA(5)-Picard-MG & 56 &3.37e-13 &  3.05e-13 &  34.99\\  
& RRE(5)-Picard-MG & 25 & 1.30e-14&  2.27e-13 &  20.57\\ 
& MPE(5)-Picard-MG & 25 &8.51e-15 &  2.27e-13 & 17.76 \\
& RRE(8)-Picard-MG & 28 & 6.12e-13 &2.27e-13&  21.14\\  
& MPE(8)-Picard-MG & 28 & 3.27e-13 &2.27e-13 & 20.29 \\  
  
\hline  

\end{tabular}
\label{table1}
\end{table}

\text{$^{a}$ Indicates that nonlinear tolerance was not attained.}

On the one hand, Table~\ref{table1} clearly shows that the convergence of the Picard-MG method deteriorates for high values of the parameter $\lambda$, i.e., when the nonlinear term in the Bratu equation dominates, we obtain poor convergence for the Picard-MG method, and we deduce the same remark for the Picard-slu method. We also notice that the convergence of Picard-MG is unstable with respect to grid size. On the other hand, the polynomial extrapolation methods accelerate considerably the convergence of
the Picard iterations; this table clearly shows the reduction of the number of iterations of Picard's iterative method when using the restarted MPE and RRE methods every $q$ steps; moreover, the CPU time to achieve the desired precision is much better for the MPE-Picard-MG and the RRE-Picard-MG methods than for Picard-MG, and we also noted that the convergence of our hybrid technique is independent of the number of points. Furthermore, we observe that the convergence results obtained by the MPE-Picard-MG and the RR-Picard-MG methods are comparable. To sum up, our technique produces better convergence results even for high $\lambda$ and $p$ values in comparison to the Anderson-accelerated Picard-MG method, and its convergence is also robust with respect to $N$.

Figures~\ref{figure3} and~\ref{figure4} are only a clear visualization of Table~\ref{table1}; in Figure~\ref{figure3} we try to show the convergence behavior of the Picard iterative method with and without extrapolation acceleration techniques for a large value of the parameter $\lambda$ and for a different number of points $N$. Moreover, we illustrate the $L_2$ norm of the error in Figure~\ref{figure4}. Furthermore, we observe that Picard's method's convergence is improved more by polynomial extrapolation techniques than by Anderson's acceleration technique.

\begin{figure}[ht!]
\centering

  \includegraphics[width=5cm,height=5cm]{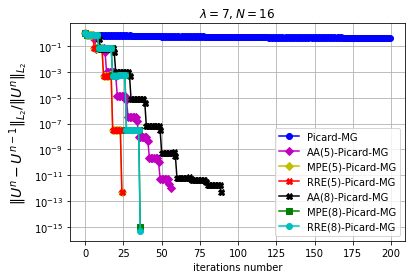}
    \includegraphics[width=5cm,height=5cm]{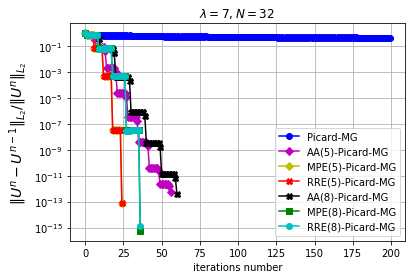}
  \includegraphics[width=5cm,height=5cm]{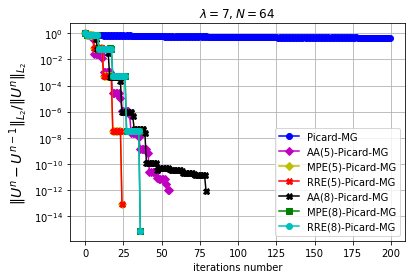}
   
    \caption{Behavior of the Picard method with and without extrapolation acceleration for $\lambda=7$ on different grid sizes}
    \label{figure3}
\end{figure}

\begin{figure}[ht!]
\centering

    \includegraphics[width=5cm,height=5cm]{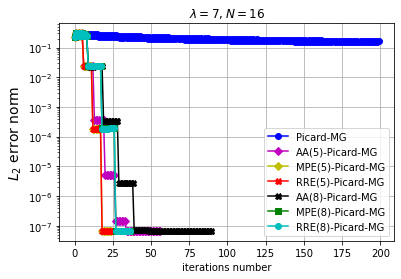}
     \includegraphics[width=5cm,height=5cm]{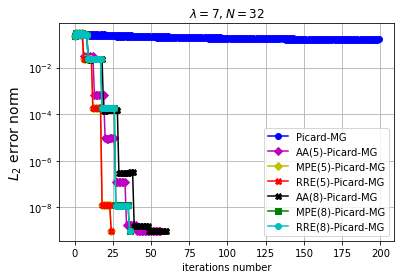}
     \includegraphics[width=5cm,height=5cm]{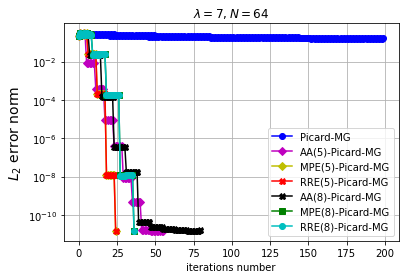}
  
    \caption{$L_2$ error norm of the Picard iterative method with and without extrapolation acceleration for $\lambda=7$ on different grid sizes}
    \label{figure4}
\end{figure}

We noticed that when we applied the MPE method and the RRE method to accelerate the iterations of Picard, we tested several values of $q$ and found that, for example, the value of $q=5$ indicates that our approach is optimal and robust (see Table~\ref{table1}). We indicate that, particularly in the nonlinear case, choosing the number of steps $q$ in the restarted extrapolation techniques is not always obvious. To illustrate these conclusions, we will vary the number of $q$ steps in the restarted MPE method. Figure~\ref{figure5} shows the behavior of the MPE-Picard-MG method for different values of $q=[2,3,4,5,6,7,8]$, where we try to save for each fixed number of points $N=[64,128]$ and for each spline degree $p=[3,4,5,6]$, the number of iterations required to achieve the desired accuracy for different values of $\lambda=[1,3,5,7,8]$.

 \begin{figure}[ht]
    \centering
\includegraphics[width=4cm,height=4cm]{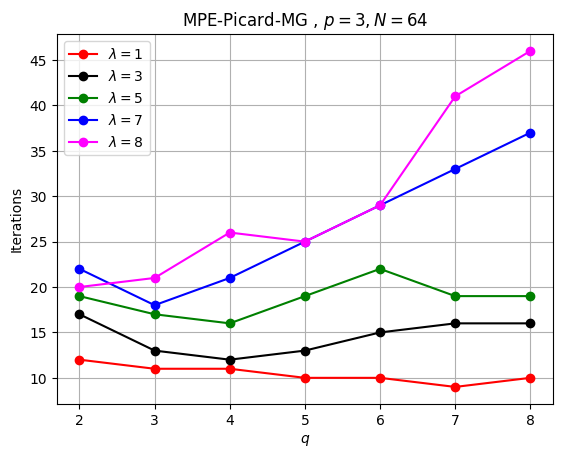}
\includegraphics[width=4cm,height=4cm]{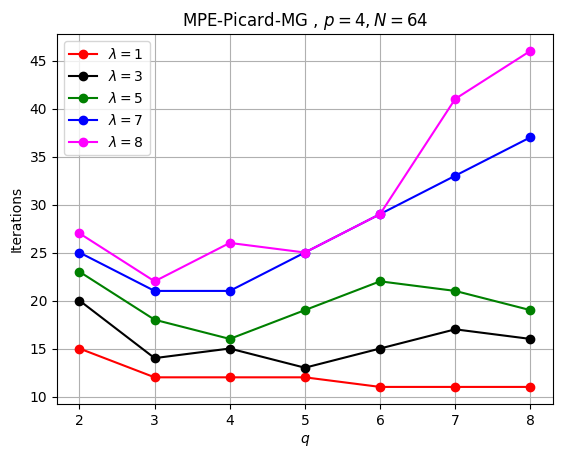}
 \includegraphics[width=4cm,height=4cm]{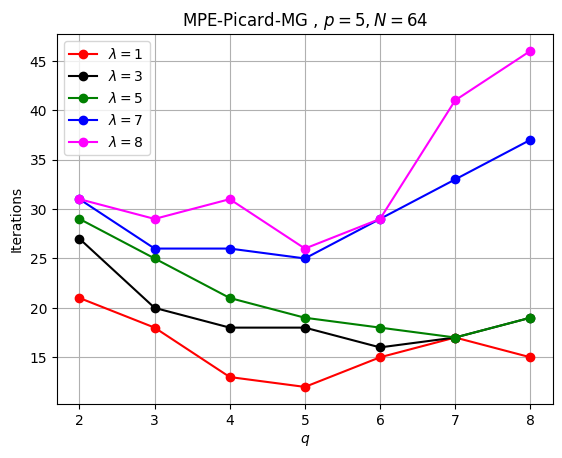}
\includegraphics[width=4cm,height=4cm]{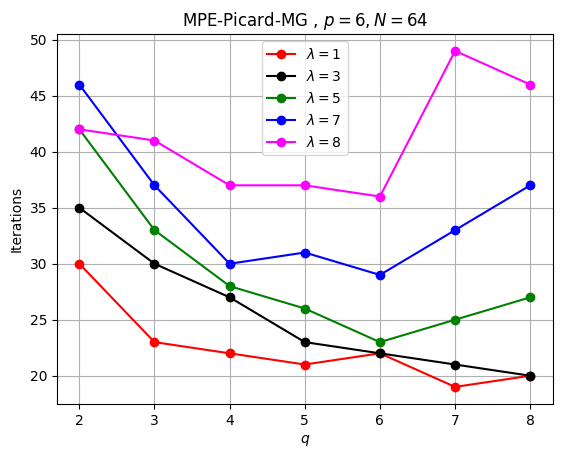}

    \includegraphics[width=4cm,height=4cm]{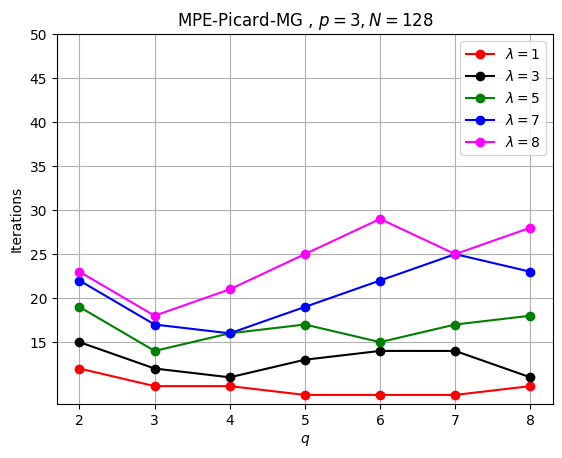}
    \includegraphics[width=4cm,height=4cm]{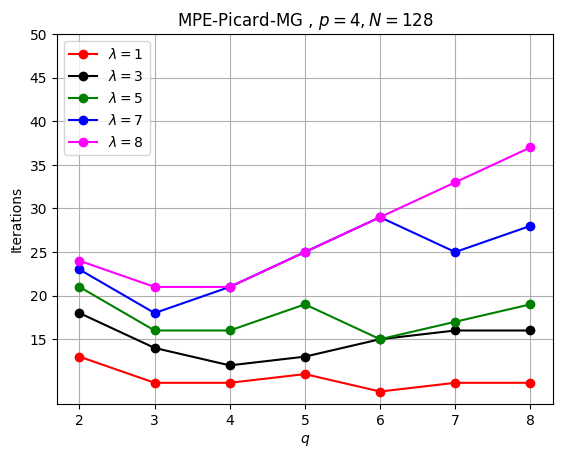}
    \includegraphics[width=4cm,height=4cm]{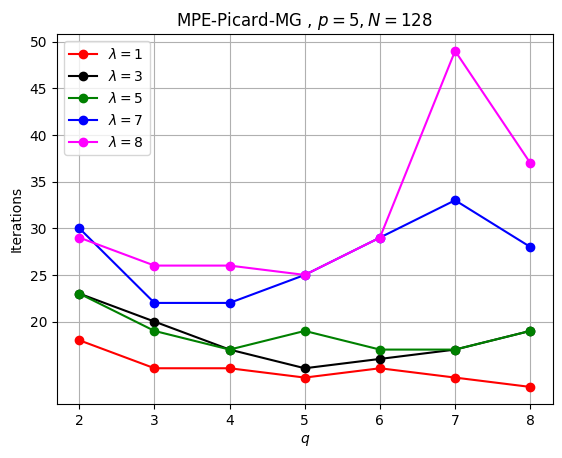}
    \includegraphics[width=4cm,height=4cm]{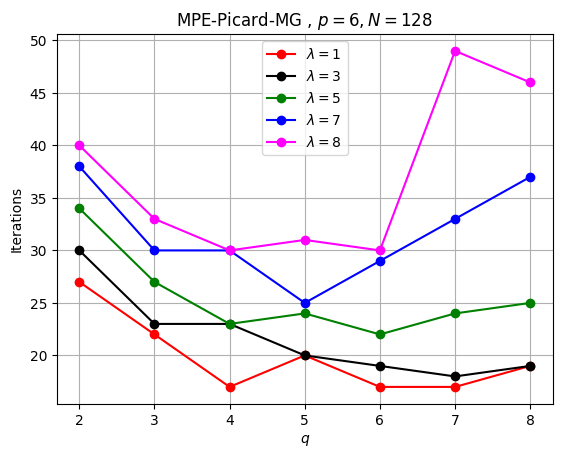}
     \caption{Behavior of the restarted MPE method were applied to Picard's iterations for different values of $q$ and $\lambda$. Top: the number of points is fixed at $N=64$. Bottom: $N=128$ }
    \label{figure5}
\end{figure}

Figure~\ref{figure5} shows that, unlike the linear case \cite{R14} (when we mentioned that the number of cycles in restarted RRE and MPE methods decreases as we increase the restart number $q$), the choice of the number of steps $q$ in the restarted MPE method is not evident; we have to choose an optimal $q$ from the numerical simulations that will make our approach optimal and robust with respect to the parameters $p$ and $N$. It's easy to see that in the nonlinear case, for high $\lambda$ values, the number of iterations of the MPE-Picard-MG method increases significantly with the number of restarts $q$. This can be explained by the fact that when $\lambda$ is large, the nonlinear term dominates Bratu's problem, and using MPE with a large number of restarts $q$  will affect the convergence of the latter since this is the same as solving the Bratu equation with the Picard iterative method, and we have shown before that Picard's convergence decreases for high values of $\lambda$. We indicate that the same convergence behavior is observed for the MPE method for different values of spline degree $p$. From this numerical result, we can choose, for example, $q=5$, since the MPE method is more stable and robust. Then, in the next numerical test example, we will illustrate the performance and efficiency of our approach, the MPE-Picard-MG method, where we will try to vary all the parameters in the $1D$ Bratu model problem: $\lambda, p$, and $N$ (see Table \textcolor{blue}{2}).

\begin{table}[ht!]

\centering
\textbf{Table 2} Restarted extrapolation (MPE) method with $q=5$ were applied to the Bratu equation at different grid sizes and for different spline degrees and $\lambda$. The number of iterations to reduce the nonlinear relative tolerance to $10^{-12}$  

\begin{tabular}{lllllllll}
\hline
\hline
\end{tabular}
\begin{adjustbox}{width=1.\textwidth}
\begin{subtable}{1.\textwidth} 
    \centering 
    \hspace*{-2. cm}
\begin{tabular}{c p{0.cm} cccc  c cccc}
\hline
 &  & \multicolumn{4}{l}{$\lambda=1$} &  & \multicolumn{4}{l}{$\lambda=3$} \\ \cmidrule(l){2-6} \cmidrule(l){7-11}
\diaghead{taun-----}{$p$}{$N$} &  & $16$   & $32$   & $64$   & $128$    & & $16$   & $32$   & $64$   & $128$    \\ \hline
$1$    & & $10$  & $11$  & $10$  & $13$  
& & $13$ & $13$  & $13$  & $13$ \\
$2$    & & $11$  & $11$   & $11$  & $11$
& & $13$  & $13$  & $13$  & $12$ \\
$3$    & & $12$  & $12$   & $10$  & $9$ 
& & $13$  & $13$  & $13$  & $13$\\
$4$    & & $13$  & $12$   & $12$   & $11$  
& & $14$  & $14$  & $13$  & $13$\\
$5$          & & $17$  & $16$   & $15$   & $14$ 
& & $19$  & $20$  & $18$  & $15$ \\
$6$     & & $21$  & $23$   & $21$   & $20$ 
& & $24$  & $26$   & $23$  & $20$ \\
\end{tabular}
\end{subtable}
\end{adjustbox}


\vspace{0.5cm}


\begin{adjustbox}{width=1.\textwidth}
\begin{subtable}{\textwidth} 
    \centering 
    \hspace*{-2. cm}
\begin{tabular}{c p{0.cm} cccc  c cccc}

\hline
 &  & \multicolumn{4}{l}{$\lambda=5$} &  & \multicolumn{4}{l}{$\lambda=7$} \\ \cmidrule(l){2-6} \cmidrule(l){7-11}
\diaghead{taun-----}{$p$}{$N$} &  & $16$   & $32$   & $64$   & $128$    & & $16$   & $32$   & $64$   & $128$   \\ \hline
$1$    & & $13$ & $18$  & $18$ & $ 13$ 
& & $19$ & $25$ & $19$ & $13$ \\
$2$    & & $19$ & $19$  & $19$  & $13$ 
& & $25$ & $25$ & $25$ & $19$ \\
$3$    & & $19$ & $19$  & $19$  & $17$ 
& & $25$ & $25$ & $25$ & $19$ \\
$4$    & & $19$ & $19$ & $19$ & $19$
& & $25$ & $25$  & $25$ & $25$ \\
$5$          & & $19$ & $20$  & $19$  & $19$ 
& & $25$ & $25$  & $25$ & $25$ \\
$6$     & & $26$  & $29$   & $26$   & $24$ 
& & $31$ & $31$  & $31$ & $25$ \\

\end{tabular}
\end{subtable}
\end{adjustbox}

\label{tab:cu-cgn-none}
\end{table}

Even for high values of $\lambda$  and in contrast to Picard's method (Picard-MG), our technique (MPE-Picard-MG) converges. We remark that there is an increasing number of iterations of the MPE-Picard-MG method as $\lambda$ grows. This suggests that convergence might be more challenging for larger values of $\lambda$. In addition, the number appears to be influenced by the spline degree $p$ \cite{R14,R40,R41,R42,R43}; however, the trend is not consistent with our solver. As a consequence, our approach is sustainably stable and robust with respect to the spline degree $p$ and the number of points  $N$.

\subsection{The $2D$ Bratu problem}  \label{sub2sec4}
Finally, in order to illustrate how extrapolation might improve the Picard approach in the $2D$ case, we test the polynomial extrapolation acceleration techniques on the $2D$ Bratu problem (\ref{eq2dbratu}) on a square domain discretized with the IGA method. The right-hand side $f$ was chosen such that the exact solution is 
$$u(x,y)= (x-x^2)(y-y^2).$$
Using a single V-cycle with weighted Jacobi with $w=2/3$ as a relaxation method for solving the linear problem in each Picard iteration, we imply the same technique as in the $1D$ example. In the following, we will illustrate the performance of our technique. Therefore, we will compare in Figure~\ref{figure6} the performances of the polynomial extrapolation methods and the Anderson extrapolation technique for accelerating the iterations of the Picard iterative method for solving the $2D$ Bratu problem. Figure~\ref{figure6} shows the results of the Picard-MG, MPE-Picard-MG, RRE-Picard-MG, and AA-Picard-MG methods when we try to report the number of iterations to achieve the desired accuracy of $10^{-08}$ and the $L_2$ norm of the error for different values of $\lambda=6.966$ (which is the critical value of the parameter $\lambda$ proposed by~\cite{R44} and $\lambda=6.808$ is a more precise critical value proposed by Glovinski, Keller, and Reinhart~\cite{R45}) and $\lambda=17$ in a fixed grid of $N=64$ points in each direction and for spline degree $p=5$.

\begin{figure}[ht!]
    \centering
   \includegraphics[width=6.5cm,height=6.2cm]{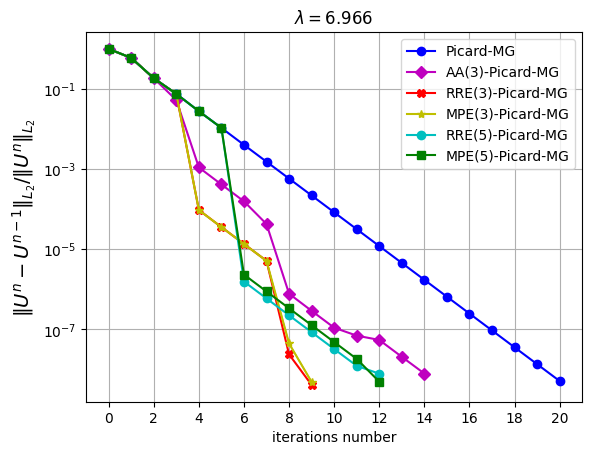}
   \includegraphics[width=6.5cm,height=6.2cm]{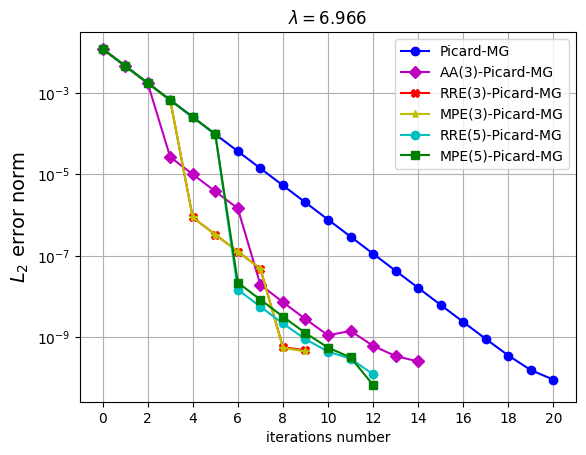}

   \includegraphics[width=6.5cm,height=6.2cm]{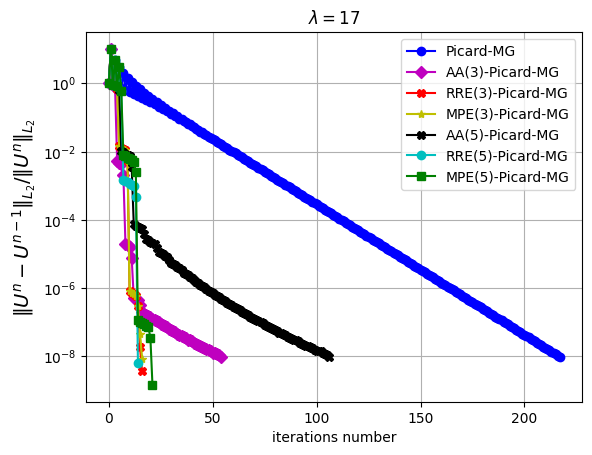 }
      \includegraphics[width=6.5cm,height=6.2cm]{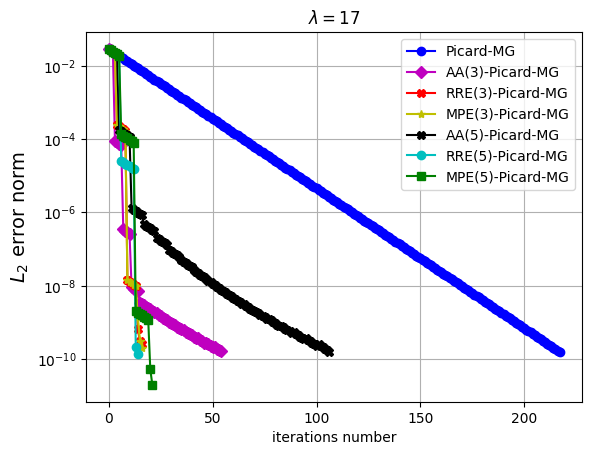 }
   \caption{ Convergence results of the Picard method with and without extrapolation acceleration techniques for the $2D $ Bratu with $\lambda=6.966 $ and $\lambda=17 $}
   \label{figure6}
\end{figure}

We observe the same kind of results as before in the $1D$ case. Polynomial extrapolation methods accelerate considerably the convergence of the Picard iterative method in this $2D$ case, especially for high values of the parameter $\lambda$. Hence, our technique is eﬃcient.

\begin{figure}[ht!]
    \centering
   \includegraphics[width=6.5cm,height=6cm]{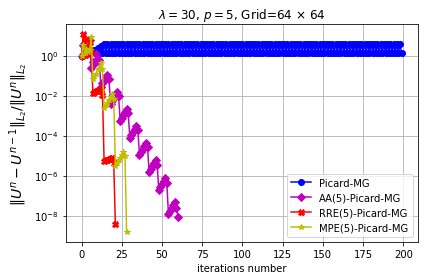}
   \includegraphics[width=6.5cm,height=6cm]{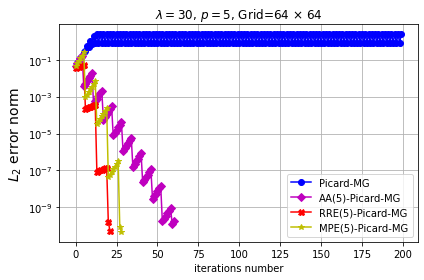}

   \caption{ Behavior of the Picard method for the $2D$ Bratu problem with and without extrapolation acceleration techniques for high values of $\lambda$}
   \label{figurelamda30}
\end{figure}
We can clearly see from Figure~\ref{figurelamda30} that when the Picard method diverges, polynomial extrapolation techniques significantly enhance Picard's convergence, and the convergence results of the two methods, RRE-Picard-MG and MPE-Picard-MG, are comparable.
Additionally, we find that polynomial extrapolation methods provide better convergence results than Anderson's acceleration method.

In the following Table~\ref{table3}, we will reproduce the convergence results for the previous methods, but this time we compute the time of execution in seconds (CPU) for each of these methods for various values of $\lambda$. It is important to remember that the codes were implemented in a sequential manner. It is clear that in this $2D$ case, the polynomial extrapolation techniques take less time to execute than Picard and Picard with the Anderson acceleration method. 

We conclude by demonstrating in Table \textcolor{blue}{4} the effectiveness and the performance of our strategy, the MPE-Picard-MG solver, for addressing the $2D$ Bratu problem. We will try to vary $\lambda, p$, and $N$, the three parameters in the $2D$ Bratu model problem, and we illustrate the number of iterations required to reach convergence up to a precision of $10^{-08}$. We remark that the iteration numbers remain uniformly bounded as we increase the parameter $\lambda$ and the spline degree $p$. Then we conclude that our technique, which combines the Picard iterative method with the polynomial extrapolation methods, is robust and optimal with respect to $\lambda$, $p$, and $N$.

\begin{table}[ht!]
    \centering  
      \caption{Extrapolation and Picard methods were applied to the $2D$ Bratu problem for different values of $\lambda$. The number of iterations and CPU time in seconds to reduce the nonlinear relative tolerance to $10^{-08}$ for spline degree $p=5$ and $N=64$}
      \label{table3}
\begin{tabular}{ |c|c|c|c|c|c| } 
\hline
 $\lambda$ & Method &iter  & $L_{2}\text{-err}$  & CPU(s)\\
\hline

\multirow{3}{4em}{3} 
&Picard-MG & 12 & 2.30e-10  & 192.48\\ 
& AA(3)-Picard-MG & 10 &  2.27e-10 &  152.68\\  
& RRE(3)-Picard-MG & 8 &  5.55e-10 &  120.33\\   
& MPE(3)-Picard-MG & 8 & 5.55e-10 & 114.91 \\  

\hline
\multirow{3}{4em}{6.966} 
&Picard-MG & 21 &  8.69e-11 & 426.45 \\  
& AA(3)-Picard-MG & 15 &  2.44e-10 & 311.75 \\   
& RRE(3)-Picard-MG  & 10 & 4.56e-10  & 224.71 \\  
& MPE(3)-Picard-MG  & 10&  4.32e-10 & 218.22 \\
& RRE(5)-Picard-MG  & 13 &  1.17e-10 & 250.23 \\  
& MPE(5)-Picard-MG  & 13&  6.46e-11 &  241.04 \\
\hline 

\multirow{3}{4em}{17}   
&Picard-MG & 218&  1.52e-10   &  3164.64  \\ 
& AA(3)-Picard-MG & 55 & 1.58e-10 & 641.06\\  
& RRE(3)-Picard-MG   & 17 & 2.50e-10& 259.05\\  
& MPE(3)-Picard-MG   & 17 & 1.92e-10 &226.43 \\  
& AA(5)-Picard-MG & 107 & 1.56e-10 & 1482.12 \\  
& RRE(5)-Picard-MG &15 & 1.31e-10 & 225.50 \\ 
& MPE(5)-Picard-MG &  22 & 1.91e-11 & 379.86\\  
\hline  
\end{tabular}
\end{table}

\begin{table}[ht!]

\centering
		\textbf{Table 4} Convergence results of the MPE method with $q=5$, applied to the $2D$ Bratu equation at different grid sizes and for different spline degrees and $\lambda$. The number of iterations to reduce the nonlinear relative tolerance to $10^{-08}$ 
  
\begin{tabular}{lllllllll}
\hline
\hline
\end{tabular}
\begin{adjustbox}{width=1.\textwidth}
\begin{subtable}{1.\textwidth} 
    \centering 
    \hspace*{-2. cm}
\begin{tabular}{c p{0.cm} cccc  c cccc}
\hline
 &  & \multicolumn{4}{l}{$\lambda=3$} &  & \multicolumn{4}{l}{$\lambda=6.966$} \\ \cmidrule(l){2-6} \cmidrule(l){7-11}
\diaghead{taun-----}{$p$}{$N$} &  & $16$   & $32$   & $64$   & $128$    & & $16$   & $32$   & $64$   & $128$    \\ \hline
$1$    & & $7$  & $7$   & $7$   & $7$  
& & $7$   & $7$    & $7$   & $7$   \\
$2$    & & $9$  & $7$   & $7$   & $7$ 
& & $12$   & $9$   & $7$   & $7$  \\

$3$    & & $11$  & $9$   & $8$  & $7$ 
& & $13$   & $12$  & $9$   & $8$ \\

$4$    & & $15$  & $10$   & $8$   & $7$  
& & $15$   & $14$   & $10$   & $9$  \\

$5$          & & $15$   & $13$    & $8$    & $8$ 
& & $18$   & $16$  & $13$   & $10$ \\

\end{tabular}
\end{subtable}
\end{adjustbox}

\vspace{.5cm}

\begin{adjustbox}{width=1.\textwidth}
\begin{subtable}{\textwidth} 
    \centering 
    \hspace*{-2. cm}
\begin{tabular}{c p{0.cm} cccc  c cccc}

\hline
 &  & \multicolumn{4}{l}{$\lambda=10$} &  & \multicolumn{4}{l}{$\lambda=17$} \\ \cmidrule(l){2-6} \cmidrule(l){7-11}
\diaghead{taun-----}{$p$}{$N$} &  & $16$   & $32$   & $64$   & $128$    & & $16$   & $32$   & $64$   & $128$   \\ \hline
$1$    & & $11$ & $13$  & $15$ & $7$
& & $15$ & $15$ & $15$ & $9$ \\
$2$    & & $15$ & $15$  & $15$  & $11$
& & $17$ & $21$ & $20$ & $15$ \\
$3$    & & $15$ & $15$  & $15$  & $15$
& & $16$ & $19$ & $20$ & $15$ \\
$4$    & & $17$ & $16$  & $15$  & $15$ 
& & $17$ & $17$ & $20$ & $21$ \\
$5$          & & $21$ & $19$  & $15$  & $16$ 
& & $26$ & $24$  & $22$  & $21$\\

\end{tabular}
\end{subtable}
\end{adjustbox}


  \label{tab:curl-cgn-none}  
\end{table}

\subsection{The Monge–Ampère equation}  
In order to demonstrate the efficiency of extrapolation methods for accelerating the convergence of the Picard iterative method combined with a multigrid solver for the resolution of nonlinear problems. We conclude by considering the numerical solution of the $2D$ elliptic standard Monge–Ampère equation (\ref{eq2dMonge_Ampere}) with non-homogeneous Dirichlet boundary conditions on the unit square.

We will focus on polynomial extrapolation techniques for the Monge-Ampère problem (\ref{eq2dMonge_Ampere}) as we have demonstrated their efficiency over Anderson acceleration through several numerical results in Subsection ~\ref{sub1sec4} and ~\ref{sub2sec4} for the $1D$ (\ref{eq1dbratu}) and $2D$ (\ref{eq2dbratu}) Bratu problems.\\
We choose a radially symmetric convex solution in $\mathbb{C}^{\infty}(\Omega)$ for the example that follows. Next, the exact solution and the right-hand side of the Monge-Ampère problem \ref{eq2dMonge_Ampere} are provided respectively, by:
$$u(x,y)= e^{\frac{x^2 +y^2}{2}},$$
and
$$f(x, y) = (1+x^2 +y^2) e^{x^2 +y^2}.$$ 
The boundary function $g$ is given as the restriction of the exact solution to the boundary $\Gamma=\partial \Omega$:
$$g= \begin{cases}
e^{\frac{x^2 }{2}},& \,\,\, 0< x< 1, \,y=0,   \\
 e^{\frac{y^2}{2}},& \,\,\, x=0, \, 0< y< 1,  \\
e^{\frac{x^2 +1}{2}},& \,\,\, 0< x< 1, \, y=1,  \\
 e^{\frac{1 +y^2}{2}},& \,\,\, x=1, \, 0< y< 1.  \\
  \end{cases}$$
Unlike the Bratu problem (\ref{eq2dbratu}), where we used a single iteration of V-cycle in each Picard iteration (see \ref{sub1sec4} and \ref{sub2sec4}), here for the resolution of Monge-Ampère problem (\ref{eq2dMonge_Ampere}) we apply the V-cycle as a solver for the elliptic problem in each iteration of Picard (\ref{eq2dMonge_Ampere_2}). We will use the V-Cycle scheme with weighted Jacobi where the relaxation parameter is $\omega = 2/3$. The following Figure~\ref{figure20} shows the convergence results of the Picard-MG method, the RRE-Picard-MG, and the MPE-Picard-MG methods with restart number $q=5$ and $q=8$ for spline degree $p=3$ at different grid sizes. We report the number of iterations to reduce the “Relative Residual" which is the $L_2$ norm of 
the difference between successive Picard iterations
$\frac{\Vert U^{n}-U^{n-1} \Vert_{L_{2}}}{\Vert U^{n} \Vert_{L_2}}$. The stopping criteria is when the accuracy of $10^{-10}$ is reached.

\begin{figure}[ht!]
\centering
    

    \includegraphics[width=5cm,height=5cm]{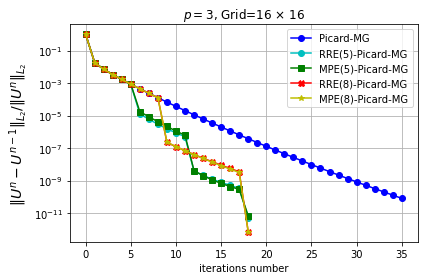}
   \includegraphics[width=5cm,height=5cm]{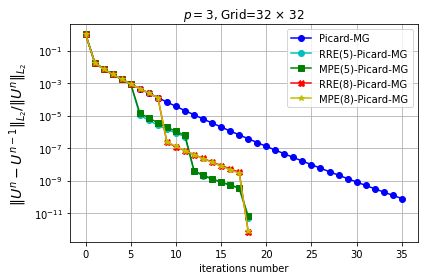}
    \includegraphics[width=5cm,height=5cm]{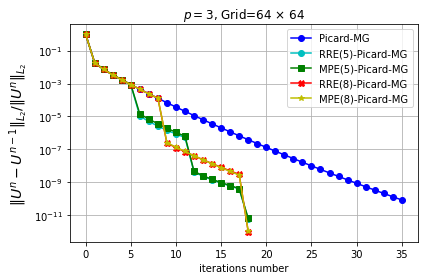}
    \caption{Convergence of the Picard method for Monge-Ampère with and without extrapolation acceleration for $p=3$ at different grid sizes}
    \label{figure20}
\end{figure}

Using the methods previously discussed, the $L_2$ norm errors ($\Vert u - u_h \Vert _{L_2(\Omega)}$), where $u_h$ is the computed approximate solution, are displayed for $p=3$ and different grid sizes in Figure \ref{figure21}.

\begin{figure}[ht!]
\centering
     \includegraphics[width=5cm,height=5cm]{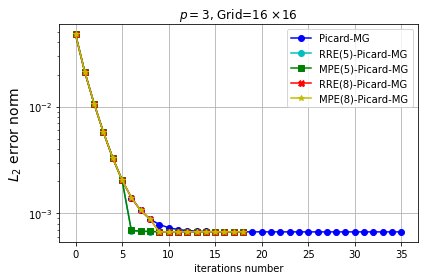}
   \includegraphics[width=5cm,height=5cm]{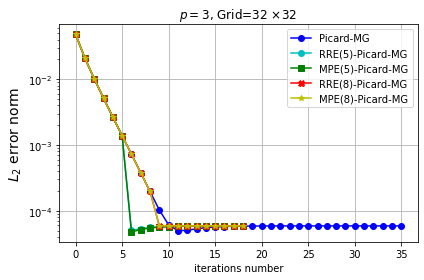}
    \includegraphics[width=5cm,height=5cm]{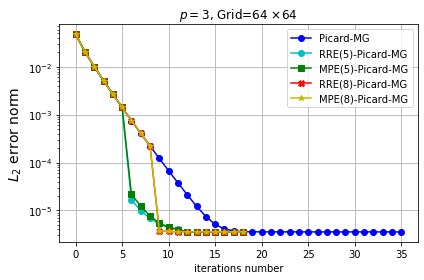}
    \caption{Error evolution of Monge-Ampère problem for spline degree $p=3$ on different grid sizes}
    \label{figure21}
\end{figure}

Figures \ref{figure20} and \ref{figure21} illustrate how well Picard's approach converges and how the errors decrease with increasing grid size. Nonetheless, we find that using polynomial extrapolation techniques on Picard allows us to obtain an adequate error in a lower number of iterations than the Picard method.

The $L_2$ norms of the error of the computed approximate solution $u_h$ of the Monge-Ampère equation with the exact solution and the CPU time are provided in Table~\ref{MA table}. In addition, “CPU(s)” indicates the total CPU time in seconds for the Picard-MG method, as well as the RRE(5)-Picard-MG and MPE(5)-Picard-MG approaches. The overall time spent assembling the right-hand side during all Picard iterations is indicated by “RHS time(s)”. In all Picard iterations, the total time required for solving the linear problem with the V-cycle solver and weighted Jacobi smoother is denoted by “MG time(s)”. The CPU time needed for the extrapolation process is referred to as “Extrapolation time(s)”. The tolerance used for the linear system in the V-cycle solver is included in the grid cell of the Table~\ref{MA table}. 

For the Monge-Ampère equation, polynomial extrapolation techniques considerably accelerate the convergence of Picard iterations. Applying the restarted MPE and RRE techniques every 5 steps clearly reduces the number of iterations required by Picard's method, as Table~\ref{MA table} shows. It also demonstrates how Picard iterations with and without extrapolation behave in terms of convergence, as well as how errors decrease with increasing grid size. When compared to Picard-MG, the CPU time required to achieve the desired precision is significantly reduced when using the MPE-Picard-MG and RRE-Picard-MG approaches. Moreover, our hybrid method seems to be robust towards grid sizes and spline degrees. Additionally, we note that the MPE-Picard-MG and RRE-Picard-MG strategies have very similar convergence results. In the current implementation, explicit matrix assembly is used in the computation and is shown to be computationally costly for large-scale problems. Future works will address this challenge in two important ways: integrating Matrix-Free techniques to avoid explicit right-hand side (RHS) assembly, which will reduce memory usage and computational costs; and parallelizing the code using OpenMP or MPI to improve scalability and efficiency, especially for multi-dimensional problems. While preserving numerical precision, these improvements aim to speed up overall efficiency.

\newpage
\begin{table}[htbp!]
    \centering
        \caption{The Monge-Ampère equation: Iteration count, error evolution, and CPU time in seconds for various spline\\ degrees at different grid sizes using Picard with and without extrapolation methods with restart number $q=5$} 
              \label{MA table}  
    \begin{tabular}{|c|c|c|c|c|c|c|c|c|}
        \hline
        Grid & Method & iter & Relative Residual & $L_2$-err & CPU(s) & RHS time(s) & MG time(s) & Extrapol time(s) \\
        \hline
        \multicolumn{9}{|c|}{\textbf{Spline degree $p=2$}} \\
        \hline
        \textbf{8 $\times$ 8} & Picard-MG & 38 & 7.42e-11 & 2.71e-02 & 4.41& 2.68 & 0.0752 & ---\\ 
                              (tol=$10^{-02}$)  & RRE(5)-Picard-MG & 19 &2.59e-11 &2.71e-02 & 2.37 &1.30 & 0.0531 & 0.000908 \\
                         & MPE(5)-Picard-MG  &  19 &  2.91e-11 &  2.71e-02 & 2.34 & 1.26 & 0.0435 & 0.000864 \\
        \hline
                   \textbf{16 $\times$ 16}  & Picard-MG          & 37 &6.77e-11 & 8.84e-03 & 11.01 & 5.20 & 0.375 & --- \\ 
                    (tol=$10^{-02}$)    & RRE(5)-Picard-MG & 19 &  2.79e-12& 8.84e-03 & 5.80 & 2.59 & 0.230 & 0.000751 \\
                          & MPE(5)-Picard-MG &  19  &  3.40e-12  &  8.84e-03  &5.95 & 2.69  & 0.190 & 0.000672   \\
            \hline
           \textbf{32 $\times$ 32}  & Picard-MG          & 37 & 7.32e-11 & 5.00e-04 & 47.66 &20.61   & 13.05  & --- \\ 
                     (tol=$10^{-03}$)   & RRE(5)-Picard-MG & 19 &6.49e-12 & 5.00e-04 &  26.92 & 10.95 & 7.47 & 0.00124 \\
                           & MPE(5)-Picard-MG & 19 &  6.96e-12 & 5.00e-04  &  25.98 & 10.84 & 6.78 & 0.000893 \\
            \hline

            \textbf{64 $\times$ 64} & Picard-MG        & 36 & 8.51e-11 &2.75e-04  &  259.38 & 85.23&  123.21 & ---\\ 
                  (tol=$10^{-03}$)    & RRE(5)-Picard-MG & 19 & 3.68e-12& 2.75e-04  &  137.95 & 43.12  & 64.87  & 0.00218  \\
                         & MPE(5)-Picard-MG &19  &  4.77e-12 & 2.75e-04 & 135.44 & 42.25  &  63.21 & 0.00165   \\
            \hline

           \textbf{128 $\times$ 128 } & Picard-MG         & 36 &  8.39e-11 & 1.05e-05 &  2812.18 & 326.96 & 2296.79  &  --- \\ 
                   (tol=$10^{-04}$)    & RRE(5)-Picard-MG & 19 & 3.40e-12& 1.05e-05  & 1456.89 & 167.13  & 1180.91  &  0.0271 \\
                       & MPE(5)-Picard-MG &19  &5.78e-12  & 1.05e-05 &  1473.27 & 171.90  & 1191.37  & 0.0200 \\
        \hline
        \hline

        \multicolumn{9}{|c|}{\textbf{Spline degree $p=3$}} \\
                \hline
 \textbf{8 $\times$ 8 } & Picard-MG      & 37 & 7.84e-11 & 1.98e-03 & 8.31 & 3.63 & 0.39 &---  \\ 
                     (tol=$10^{-02}$)  & RRE(5)-Picard-MG  & 19 &5.99e-12 &1.98e-03 & 4.36 &1.80 &0.203 & 0.000708 \\
                         & MPE(5)-Picard-MG  &  19 &  4.83e-12  &  1.98e-03 & 4.40 & 1.85 & 0.192 & 0.000612  \\
            \hline
         \textbf{16 $\times$ 16}    & Picard-MG      & 36 &8.14e-11 & 6.72e-04 & 25.54 &14.12 & 1.41 & --- \\ 
                    (tol=$10^{-03}$)   & RRE(5)-Picard-MG  & 19 &  5.32e-12& 6.72e-04 & 13.66 &  7.14 &  0.81 & 0.000794  \\
                        & MPE(5)-Picard-MG  &  19  &  7.01e-12  &  6.72e-04  &13.99  & 7.36 & 0.79 &   0.000682 \\
            \hline
         \textbf{32 $\times$ 32}    & Picard-MG       & 36 & 8.06e-11 & 5.91e-05 & 105.05  & 56.42 &  16.26 & --- \\ 
                    (tol=$10^{-04}$)   & RRE(5)-Picard-MG  & 19 &5.25e-12 & 5.91e-05 &   59.45 & 30.29 & 9.59 & 0.00104 \\
                          & MPE(5)-Picard-MG  & 19 &  6.73e-12 & 5.91e-05  &  56.61 & 29.28 & 8.11 &  0.000838 \\
            \hline

        \textbf{ 64 $\times$ 64}    & Picard-MG        & 36 & 8.24e-11 &3.54e-06 &  580.14 & 231.42 & 222.09 & --- \\ 
                  (tol=$10^{-05}$)   & RRE(5)-Picard-MG  & 19 & 5.59e-12& 3.54e-06  &  304.02 & 117.71 & 112.92 & 0.00201 \\
                         & MPE(5)-Picard-MG  &19  &  6.26e-12 & 3.54e-06 & 321.22  &123.54 & 119.78 & 0.00194  \\
            \hline

         \textbf{128 $\times$ 128}    & Picard-MG        & 36 &  8.27e-11 & 8.21e-07 &  4904.35 & 951.45 &  3447.58 & --- \\ 
                  (tol=$10^{-06}$)   & RRE(5)-Picard-MG  & 19 & 7.24e-12& 8.21e-07  & 2630.99 & 493.40 & 1815.14 &  0.0414 \\
                       & MPE(5)-Picard-MG  &19  &8.81e-12  & 8.21e-07 &  2721.56  & 508.94 & 1893.77 & 0.0248  \\
            \hline
         \hline

         \multicolumn{9}{|c|}{\textbf{Spline degree $p=4$}} \\
        \hline
          \textbf{8 $\times$ 8 }   & Picard-MG       & 37 & 8.67e-11 & 3.54e-05 & 17.87 & 8.49   & 2.38  & ---     \\ 
                (tol=$10^{-04}$)     & RRE(5)-Picard-MG & 19 &6.70e-11 &3.54e-05 & 9.52     & 4.40   & 1.22 & 0.000738 \\
                         & MPE(5)-Picard-MG&  19 &  5.37e-11  &  3.54e-05 & 9.75    & 4.46   & 1.39 & 0.000613 \\
            \hline
        \textbf{16 $\times$ 16}     & Picard-MG      & 36 &8.27e-11 & 3.78e-06 & 67.01 & 35.24  & 10.88 & ---    \\ 
                  (tol=$10^{-05}$)     & RRE(5)-Picard-MG & 19 &  6.62e-12& 3.78e-06 & 33.87    & 17.24   & 4.91 & 0.000833 \\
                       & MPE(5)-Picard-MG &  19  &  9.49e-12  & 3.78e-06  & 34.99    & 18.13  & 4.67 & 0.000665   \\
            \hline
           \textbf{32 $\times$ 32}  & Picard-MG        & 36 & 8.36e-11 & 3.03e-06 & 326.27 & 150.04   & 94.39 & ---     \\ 
                  (tol=$10^{-06}$)    & RRE(5)-Picard-MG & 19 & 4.64e-12 & 3.03e-06 &  169.68    & 76.43   & 45.38& 0.00106 \\
                         & MPE(5)-Picard-MG & 19 &  7.03e-12 & 3.03e-06  &  155.28    & 71.75   & 39.53 & 0.000920 \\
            \hline

          \textbf{ 64 $\times$ 64}  & Picard-MG        & 36 & 8.24e-11 &7.17e-07  &  1570.53 & 548.08  & 733.29 & ---    \\ 
                   (tol=$10^{-07}$)    & RRE(5)-Picard-MG & 19 & 5.35e-12& 7.17e-07 &  844.20     & 286.55  & 383.31 & 0.00208 \\
                           & MPE(5)-Picard-MG &19  &  6.80e-12 & 7.17e-07 & 831.90      & 282.18   & 378.64 & 0.00193 \\
            \hline

            \textbf{ 128 $\times$ 128 }& Picard-MG       & 36 & 8.19e-11  & 2.53e-07 & 9769.46  & 2256.38  & 6370.82 & ---    \\ 
                      (tol=$10^{-08}$)    & RRE(5)-Picard-MG & 19 & 3.43e-12& 2.53e-07   & 4597.89     & 1176.27  & 3009,58 & 0.0796 \\
                          & MPE(5)-Picard-MG & 19 & 5.72e-12 & 2.53e-07 &  4491.40    & 1113.51 & 3054,09 & 0.0616 \\
            \hline
    \end{tabular}%
   
\end{table}

\section{Conclusion}
Through this paper, we are able to build a new solver: a hybrid approach that combines the Picard iterative method and the polynomial type extrapolation methods (MPE-Picard-MG and RRE-Picard-MG) for the resolution of nonlinear problems such as the nonlinear eigenvalue Bratu problem and the Monge-Ampère equation, which is concerned with using the restarted version of the extrapolation techniques to accelerate the iterations of the Picard method with multigrid as a linear solver. The motivation behind that is well known: the convergence of the Picard method stagnates and, in most cases, diverges, so coupling polynomial extrapolation techniques with Picard improves Picard's convergence. We have performed a numerical analysis showing the performance of our solver compared to the Picard method without extrapolation acceleration techniques and the Picard method accelerated by the Anderson method for solving the Bratu problem in one and two dimensions and the two-dimensional Monge-Ampère equation. Therefore, using just a single V-cycle in each iteration of the Picard iterative method for the Bratu problem, a robust convergence is established numerically for our approaches (RRE-Picard-MG and MPE-Picard-MG) with respect to the parameters $\lambda$ in the Bratu equation, the spline degree $p$, and the number of points $N$, while the Picard method diverges, especially for high values of the parameter $\lambda$. As a consequence, compared to Picard's method without extrapolation (Picard-MG) or even Picard accelerated using Anderson acceleration (AA-Picard-MG), we achieve better convergence outcomes and a significant reduction in CPU time for our approaches. When solving the Monge-Ampère equation, we use the V-cycle as a solver for the elliptic problem in every Picard iteration, as opposed to the Bratu problem, when we only used a single iteration of the V-cycle in every Picard iteration, and we have clarified that polynomial extrapolation techniques provide superior convergence results than the Picard iterative method.






\end{document}